\documentclass[12 pt]{amsart}
\usepackage{amsmath,amssymb,amsthm,hyperref}
\usepackage{longtable,}
\usepackage{a4wide}
\usepackage{color}
\usepackage{enumerate}
\usepackage{cleveref}

\usepackage{tikz}
\usetikzlibrary{decorations.pathreplacing}
\usetikzlibrary{shapes,snakes}

\newtheorem{thm}{Theorem}[section]
\newtheorem*{thm*}{Theorem}
\newtheorem*{conj*}{Conjecture}
\newtheorem{cor}[thm]{Corollary}

\newtheorem{lem}[thm]{Lemma}
\newtheorem{prop}[thm]{Proposition}

\theoremstyle{remark}

\theoremstyle{definition}

\newtheorem{defn}[thm]{Definition}
\newtheorem{prob}[thm]{Problem}
\newcounter{claim}[thm]

\title{Most Generalized Petersen graphs of girth 8 have cop number 4}

\author{Joy Morris}
\author{Tigana Runte}
\author{Adrian Skelton}

\thanks{This work was supported by the Natural Science and Engineering Research Council of Canada (grant RGPIN-2017-04905), by the University of Lethbridge Chinook Research Awards program, and by the NSERC Undergraduate Research Awards program.}

\address{Department of Mathematics and Computer Science\\
University of Lethbridge\\
Lethbridge, AB T1K 3M4\\
Canada}

\email{joy.morris@uleth.ca}
\email{tigana.runte@uleth.ca}
\email{adrian.skelton@uleth.ca}
\subjclass[2020]{05C57, 91A46}
\begin{document}

\begin{abstract}
A generalized Petersen graph $GP(n,k)$ is a regular cubic graph on $2n$ vertices (the parameter $k$ is used to define some of the edges). It was previously shown (Ball et al., 2015) that the cop number of $GP(n,k)$ is at most $4$, for all permissible values of $n$ and $k$.
In this paper we prove that the cop number of ``most" generalized Petersen graphs is exactly $4$. More precisely, we show that unless $n$ and $k$ fall into certain specified categories, then the cop number of $GP(n,k)$ is $4$. The graphs to which our result applies all have girth $8$.

In fact, our argument is slightly more general: we show that in any cubic graph of girth at least $8$, unless there exist two cycles of length $8$ whose intersection is a path of length $2$, then the cop number of the graph is at least $4$. Even more generally, in a graph of girth at least $9$ and minimum valency $\delta$, the cop number is at least $\delta+1$.
\end{abstract}

\maketitle

\section{Introduction}

A robber is on the loose and you need to determine how many cops are needed to ensure his capture (a question of fiscal restraint!). Cops and robbers is a pursuit-evasion game played on graphs \cite{nowa, quilliot} by two players on a simple graph $G$. The game starts with the cop player placing up to her allowed number of cops on her choice of vertices in $G$, followed by the robber placing his single (in the original version of the game, which is what we will be considering) pawn on his choice of vertex in $G$. Both players are fully aware of the structure of the graph and the positions of all the pawns, and take turns with the cop player having the first play. On their turn, the players may elect to have their pawn(s) remain static at the current position by passing play, or to move to one of its neighbouring vertices. The cop player may move all of her pawns in a single turn. Cops may also congregate at the same vertex (imagine, if you will, a corner donut shop), but this may not have any advantage during the course of the game. The game ends when a cop occupies the same vertex as the robber -- the robber is thus captured, and the cop player wins. If the robber has a strategy that will enable him to evade the cops forever, then he wins. The question to consider is this: how many cops must be engaged in play to guarantee that the cops can win?

\begin{defn}
The \emph{cop number} of a graph $G$, denoted $c(G)$, is the smallest positive integer $k$ such that k cops suffice to capture the robber in a finite number of moves played on the graph $G$.
\end{defn}

An interesting family of graphs to play the game on, and the family of graphs that we will be discussing throughout this paper, is the generalized Petersen graph family.

\begin{defn}
A \emph{generalized Petersen graph} $GP(n,k)$ is a graph with vertex set $$\{a_{0}, a_{1},\ldots,a_{n-1}, b_{0}, b_{1},\ldots, b_{n-1}\}$$ and edge set $$\{a_{i}a_{i+1}, a_{i}b_{i}, b_{i}b_{i+k} : 0 \le i \le n-1\}$$ where subscripts are read modulo $n$, $n \ge 5$, and $k< n/2$.
\end{defn}

The assumption that $k \neq n/2$ ensures that generalized Petersen graphs are always cubic. The fact that the same graph would be produced as $GP(n,k)$ and $GP(n,n-k)$ allows the assumption that $k<n/2$. Note that the Petersen graph is $GP(5,2)$.

\begin{defn}
The \emph{girth} of a graph is the length of a shortest cycle contained in the graph.
\end{defn}

In this paper we will be focusing on generalized Peterson graphs with a girth of 8; we will explore the relationship between generalized Peterson graphs of girth 8 and their cop numbers.

For the purpose of this paper, we will separate the vertex set into two separate sets, $A$ and $B$, which are the vertices labelled $a_i$ in our definition (generally drawn as the outer cycle) and the vertices labelled $b_i$ (generally drawn as the inner cycle), respectively.

In \cite{ball}, Ball et al.~showed that for any generalized Petersen graph $GP(n,k)$, the cop number is at most 4. The goal of this paper is to show that many generalized Petersen graphs of girth 8 have a cop number of exactly $4$, and classify possible exceptions. 

Since the parameters of generalized Petersen graphs of girth $8$ are understood (see Table~\ref{tab:girth}), our concluding result will be a reasonably short list of families of parameters that includes all generalized Petersen graphs (up to isomorphism) that do \emph{not} have a cop number of $4$. This list includes all generalized Petersen graphs that do not have a girth of $8$. This result is presented in Corollary~\ref{cor:final}.

\section{Previous research}

In this study, we have applied a classic version of cops and robber, whereby moves are limited and the location of the players are visible to their opponents with perfect information (as if there were two helicopters perched over the neighbourhood reporting relative locations to the escaped robber and the pursuing cops). Cops and robbers, however, is a game with many different versions all with different rules. Some variations include games played without perfect information \cite{isler}, ``lazy" cops and robbers \cite {bonato}, or traps \cite{clarke}.

In addition to varying rules of play, there are also varying types of graphs upon which the game can be played. We focus on generalized Peterson graphs in this paper, although our main result applies to any cubic graph of girth $8$. In \cite{ball} the analysis was extended to I graphs. They proved that the cop number of a connected I-graph $I(n,k,j)$ is less than or equal to 5.

For generalized Petersen graphs, most research on the cop number has focussed on its relation to the parameters $n$ and $k$, rather than focussing on the girth of the graph.  However, the girth of a generalized Petersen graph is straightforward to determine from $n$ and $k$. Most generalized Petersen graphs have girth 8, and therefore fall within the scope of this paper, but there are infinite families whose girth is $g$ for each $3 \le g \le 7$. We summarize the information about this relationship in Table~\ref{tab:girth}. In understanding this table, it is important to be aware of the isomorphism classes of generalized Petersen graphs.

\begin{prop}[Steimle and Staton, \cite{steimle}]\label{prop:iso}
Two generalized Petersen graphs $GP(n,k)$ and $GP(n,\ell)$ are isomorphic if and only if $k=\ell$ or $k\ell\equiv \pm 1\pmod{n}$. 
\end{prop}

In Table~\ref{tab:girth}, only the smallest value of $k$ in any given isomorphism class is listed. For example, when $n=2k+1$ the graph $GP(n,k)$ has girth $5$, but this relation is not included in the table because the corresponding graph is isomorphic to $GP(n,2)$. (Taking $\ell=2$ and $n=2k+1$ gives $k\ell\equiv -1 \pmod{n}$.)
 
\begin{table}[h!]
  \begin{center}
    \caption{Girth table \cite[Theorem 5]{boben}.}
    \label{tab:girth}
    \begin{tabular}{l|l|l|l|l|l}
      \textbf{Girth 3} & \textbf{Girth 4} & \textbf{Girth 5} & \textbf{Girth 6} & \textbf{Girth 7} & \textbf{Girth 8}\\

      \hline
	
      $n=3k$ & $n=4k$ & $n=5k$ & $n=6k$ & $n=7k$ & otherwise \\
      & $k=1$ & $k=2$ & $k=3$ & $k=4$ \\
      & & $n=5k/2$ & $n=2k+2$ & $n=7k/2$  \\
      & & & & $n=7k/3$ \\
	  & & & & $n=2k+3$ \\
	  & & & & $n=3k\pm2$\\
    \end{tabular}
  \end{center}
\end{table}

For some families that do not have girth 8, the cop number is already known, or tighter bounds have been found. For example, when $k=1$, $GP(n,1)$ has a girth of 4, and its cop number is known to be 2 \cite{ball}. Likewise, when $k=3$ the girth is $6$ (except for some small values of $n$) and the cop number is known to be at most $3$ \cite{ball}.  

It is well-known that a graph has a cop number of $1$ if and only if it has a \emph{pitfall}, also known as a \emph{corner} (see for example \cite[pp.~30--33]{anthony}). (A \emph{pitfall} is a vertex whose neighbourhood is dominated by a neighbouring vertex.) No generalized Petersen graph contains a pitfall, so the cop number of any generalized Petersen graph is at least $2$. In order for a generalized Petersen graph to have a cop number of $2$, it must have girth $3$ or $4$, since for a cubic graph having girth $5$ or more yet cop number $2$ would contravene the bound of Frankl \cite{frankl}. For the sake of completeness, we will give a brief explicit proof of this in our situation in Section~\ref{sec:2-cops}. This means that when it is proved in \cite{ball} that the cop number is at most $3$, it will actually be $3$ unless the girth is $4$ or less. So in fact, when $k=3$ the cop number of $GP(n,k)$ is $3$ unless $n \in \{9,12\}$.

Although they do not mention this, the argument given in \cite{ball} for $k=3$ also serves to show that the cop number is at most $3$ when $k=2$ (in which case the girth is $5$ except for some small values of $n$). We present this result here.

\begin{prop}\label{prop:k=2}
The cop number of $GP(n,2)$ is at most $3$. In fact, as long as the girth of this graph is $5$ (that is, unless $n \in \{6,8\}$), the cop number of $GP(n,2)$ is exactly $3$.
\end{prop}

\begin{proof}
This is very similar to the proof of \cite[Theorem 5.1]{ball}. One cop can guard the isometric path in $GP(\infty,2)$ having vertex set $\{b_0,a_0,a_1,b_1\}$. This disconnects the remainder of $GP(\infty,2)$. While one cop plays the lifted strategy of guarding this tree, the other two cops play the lifted strategy of pushing the robber to the right, as in \cite[Corollary 2.2]{ball}. This shows that three cops suffice, so the cop number is at most $3$. The lower bound of Frankl \cite{frankl} suffices to show that the cop number is at least $3$ when the girth is $5$.
\end{proof}

In Table~\ref{tab:table2}, we list all generalized Petersen graphs with $n \le 40$ whose cop number achieves the upper bound of $4$ given in \cite{ball}. (Our table is based on the table in \cite{ball}.) There are $60$ such graphs, $57$ of which have girth $8$. Note that of these $60$ graphs, only $GP(28, 8)$, $GP(35, 10)$, and $GP(35, 15)$ have girth $7$. We remark that two of the sets of parameters that appear in Table~\ref{tab:table2} did not appear in \cite{ball}; however, the source code created and used by the authors of \cite{ball} is freely available \cite{bell} and we used it to verify that these two sets of parameters do indeed produce graphs whose cop number is $4$. The parameters missing from their paper are $n=25$ with $k=7$ and $n=40$ with $k=17$. A list of the cop number of every generalized Petersen graph with $n \le 30$ is also given in \cite{burgess}, and their list also shows the cop number of $GP(25,7)$ being $4$.

\begin{table}[h!]
  \begin{center}
    \caption{$GP(n,k)$ with cop number 4, for $n \le 40$ \cite{ball} \cite{burgess}}
    \label{tab:table2}
    \begin{tabular}{l|l|l||l|l|l}
      \mathversion{bold}$n$ & \mathversion{bold}$k$ & \textbf{girth} & \mathversion{bold}$n$ & \mathversion{bold}$k$ & \textbf{girth}\\
      
      \hline
      25 & 7 & 8 & 34 & 6, 10, 13, 14&8, 8, 8, 8\\
      26 & 10 & 8 & 35 & 6, 8, 10, 13, 15&8, 8, 7, 8, 7\\
      27 & 6 & 8       & 36 & 8, 10, 14, 15&8, 8, 8, 8\\
      28 & 6, 8 &8, 7& 37 & 6, 7, 8, 10, 11, 14, 16&8, 8, 8, 8, 8, 8, 8\\
      29 & 8, 11, 12 &8, 8, 8 & 38 & 6, 7, 8, 11, 14, 16&8, 8, 8, 8, 8, 8\\
      31 &  7, 9, 12, 13 &8, 8, 8, 8 & 39 & 6, 7, 9, 11, 15, 16, 17&8, 8, 8, 8, 8, 8, 8\\
      32 & 6, 7, 9, 12  &8, 8, 8, 8 & 40 & 6, 7, 9, 11, 12, 15, 17&8, 8, 8, 8, 8, 8, 8\\
      33 & 6, 7, 9, 14 &8, 8, 8, 8 &&&
    \end{tabular}
  \end{center}
\end{table}

We have presented here only a very limited description of some of the extensive research relating to the game of cops and robbers, and to generalized Peterson graphs.

\section{Cop number $2$}\label{sec:2-cops}

Before we begin stating our results, we must define some terminology:

\begin{defn}
The cops have \emph{trapped} the robber if for each edge incident to the robber there is a cop at distance no more than 2 away from the robber along a path that uses that edge.
\end{defn}

A visual representation of this situation is shown in Figure~\ref{fig:trapped}, using the distance-$2$ subgraph of the robber's vertex. (In most of this paper, since we are assuming that the girth is 8, this subgraph is a tree.) As illustrated by the figure, if one cop is in each of the three branches, occupying any one of the three vertices within distance $2$ of the robber, the robber will lose in at most two moves. (We are assuming that the cops play optimally. It will take two moves only if all cops all start at distance 2 and the robber passes).

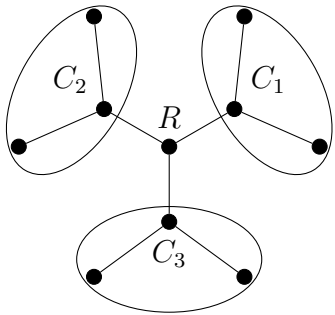
\begin{figure}
\caption{These cops have trapped the robber.}\label{fig:trapped}
\begin{center}
\begin{tikzpicture}
\node[label=above:$R$] (P0) at (0:0cm) [circle,draw,fill=black!100, inner sep=2pt, minimum width=5pt] {};
\node[label=above right:$C_1$] (P1) at (30:1cm) [circle,draw,fill=black!100, inner sep=2pt, minimum width=5pt] {};
\node[label=above left:$C_2$] (P2) at (150:1cm) [circle,draw,fill=black!100, inner sep=2pt, minimum width=5pt] {};
\node[label=below:$C_3$] (P3) at (270:1cm) [circle,draw,fill=black!100, inner sep=2pt, minimum width=5pt] {};
\node (P11) at (0:2cm) [circle,draw,fill=black!100, inner sep=2pt, minimum width=5pt] {};
\node (P12) at (60:2cm) [circle,draw,fill=black!100, inner sep=2pt, minimum width=5pt] {};
\node (P21) at (120:2cm) [circle,draw,fill=black!100, inner sep=2pt, minimum width=5pt] {};
\node (P22) at (180:2cm) [circle,draw,fill=black!100, inner sep=2pt, minimum width=5pt] {};
\node (P31) at (240:2cm) [circle,draw,fill=black!100, inner sep=2pt, minimum width=5pt] {};
\node (P32) at (300:2cm) [circle,draw,fill=black!100, inner sep=2pt, minimum width=5pt] {};
\draw (P11)--(P1)--(P0)--(P2)--(P21);
\draw (P1)--(P12);
\draw (P2)--(P22);
\draw (P0)--(P3)--(P31);
\draw (P3)--(P32);
\draw [rotate=30] (0:1.5cm) ellipse (20pt and 35pt);
\draw [rotate=150] (0:1.5cm) ellipse (20pt and 35pt);
\draw [rotate=270] (0:1.5cm) ellipse (20pt and 35pt);
\end{tikzpicture}
\end{center}
\end{figure}

So if a robber is not trapped, he has a legal move (other than passing) that does not result in him being caught on the cops' turn. Thus, if there is no configuration in which the cops can trap the robber, then the cops cannot win, since the robber can on each turn take the legal move that does not result in his being caught. 

In order for two cops to be able to trap a robber, then, some of the vertices in Figure~\ref{fig:trapped} must be identified. This means that the girth must be $4$ or less. Putting this (and our previous observation that generalized Petersen graphs do not have pitfalls) together with the information from Table~\ref{tab:girth}, we obtain the following result.

\begin{prop}
A generalized Petersen graph cannot have cop number $1$. A generalized Petersen graph can have cop number $2$ only if $k=1$, $n=3k$, or $n=4k$.
\end{prop}

As previously mentioned, it has been observed that when $k=1$ the cop number is indeed $2$, and this is also true when $n=3k$ or $n=4k$ and $k \in \{2,3\}$. However, when $k=4$ and $n=3k$ or $n=4k$ the cop number is $3$. All of these results are found in \cite{ball}.

\section{Trapping the robber with three cops}\label{sec:trapped}

We can  extend the definition of a trapped robber to apply to the situation that occurs immediately before the robber is trapped:

\begin{defn}
The cops have \emph{$2$-trapped} the robber if the robber's best possible move (aside from passing) allows the cops to trap him with their move. This can be generalized to $n-$trapped, for $n>2$: the cops have $n$-trapped the robber if the robber's best possible move (aside from passing) allows the cops to $(n-1)$-trap him with their move. 
\end{defn}

We ignore passing as one of the robber's options in this definition only because it complicates descriptions without affecting the underlying situation. As in our definition of trapped, it may be possible for the robber (by passing) to remain $n$-trapped for an additional turn, before becoming $(n-1)$-trapped, but the robber is then
 forced to either move or be caught.

With these definitions in hand, we state our key lemma. This will be the essential ingredient that we use to prove that many generalized Petersen graphs of girth 8 have cop number $4$.

\begin{lem}\label{lem:main}
Suppose that we play cops and robbers on a cubic graph $G$ of girth at least $8$. Further suppose that we are playing with three cops.

The only configuration in which the cops have $2$-trapped the robber but have not trapped the robber can be described as follows (note that the graph must have girth $8$ for this to be possible):
\begin{itemize}
\item two of the cops are sitting at points antipodal to the robber on cycles of length $8$;
\item these two cycles have as their intersection a path of length two consisting of two of the three edges incident with the robber's vertex; and
\item the third cop is at distance $1$ or $2$ from the robber, on a path that includes the third edge incident with the robber's vertex.
\end{itemize}
\end{lem}

Before proving this lemma, we provide some more commentary.

Figure~\ref{fig:2-trapped} shows a configuration in which the cops have $2$-trapped the robber. (Note that the dashed lines do not represent edges. Instead they signify that the two vertices on each end are actually the same vertex.) Cops $C_2$ and $C_3$ are each at the antipodal point of a cycle of length $8$ from the robber, and the intersection of these two cycles is the path of length $2$ consisting of the edges $e_2$ and $e_3$. (We do not rule out the possibility that some of the other vertices at distance $4$ from the robber are also in fact identified.)

As you can see, the robber in Figure~\ref{fig:2-trapped} is $2$-trapped. He cannot move to $v_1$ without being caught by cop $C_1$. His options if he moves to $v_2$ or $v_3$ are analogous to each other, so we will only discuss the case where he moves to $v_2$. The cop $C_1$ follows the robber, maintaining its current distance of $1$ or $2$. Cops $C_2$ and $C_3$ come down the middle branch of the tree toward $v_2$. Since they started at a distance of $4$ from the robber and both are moving toward him (and he also moved toward them), at the end of their turn they are each at distance $2$, one on the left branch from $v_2$ and the other on the right branch (while $C_1$ is on the third branch from $v_2$: the one that includes edge $e_2$), so the cops have trapped the robber. (The robber can only afford to pass without being caught if cop $C_1$ begins at distance $2$ from him. In this case, cop $C_1$ moves to $v_1$ and the other two cops remain in their current positions. On the robber's next turn, he must move to $v_2$ or $v_3$ to avoid being caught by cop $C_1$, and our previous analysis is unchanged.)

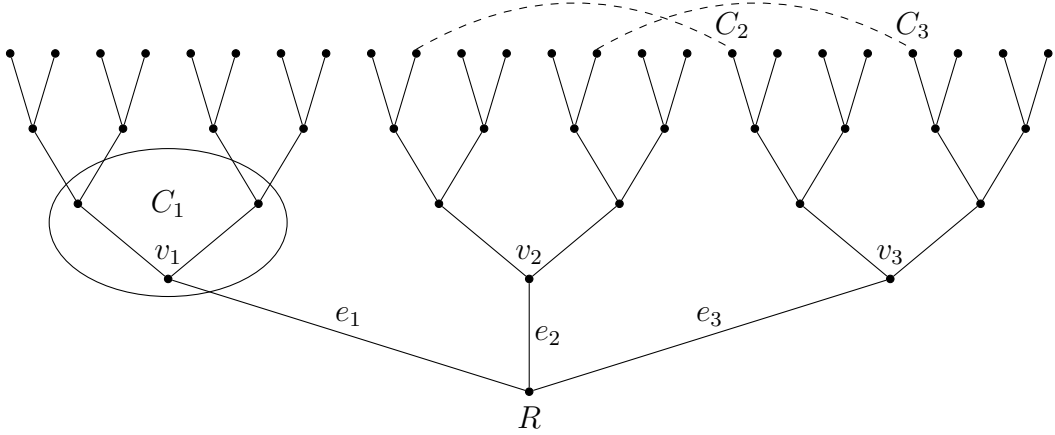
\begin{figure}\caption{These cops have $2$-trapped the robber. (Vertices joined by dashed edges are identified, and $C_1$ is on any one of the encircled vertices)}\label{fig:2-trapped}
	\begin{tikzpicture}
	\node at (-2.4,.5) {$e_1$};
	\node at (.25,.25) {$e_2$};
	\node at (2.4,.5) {$e_3$};
	\node at (-4.8,2) {$C_1$};
	\node[label=below:$R$] (P0) at (0,-.5) [circle,draw,fill=black!100, inner sep=1pt, minimum width=1pt] {};
	\node[label=above:$v_1$] (P1) at (-4.8,1) [circle,draw,fill=black!100, inner sep=1pt, minimum width=1pt] {};
	\node[label=above:$v_2$] (P2) at (0,1) [circle,draw,fill=black!100, inner sep=1pt, minimum width=1pt] {};
	\node[label=above:$v_3$] (P3) at (4.8,1) [circle,draw,fill=black!100, inner sep=1pt, minimum width=1pt] {};
	\node (P11) at (-6,2) [circle,draw,fill=black!100, inner sep=1pt, minimum width=1pt] {};
	\node (P12) at (-3.6,2) [circle,draw,fill=black!100, inner sep=1pt, minimum width=1pt] {};
	\node (P21) at (-1.2,2) [circle,draw,fill=black!100, inner sep=1pt, minimum width=1pt] {};
	\node (P22) at (1.2,2) [circle,draw,fill=black!100, inner sep=1pt, minimum width=1pt] {};
	\node (P31) at (3.6,2) [circle,draw,fill=black!100, inner sep=1pt, minimum width=1pt] {};
	\node (P32) at (6,2) [circle,draw,fill=black!100, inner sep=1pt, minimum width=1pt] {};
	\node (P111) at (-6.6,3) [circle,draw,fill=black!100, inner sep=1pt, minimum width=1pt] {};
	\node (P112) at (-5.4,3) [circle,draw,fill=black!100, inner sep=1pt, minimum width=1pt] {};
	\node (P121) at (-4.2,3) [circle,draw,fill=black!100, inner sep=1pt, minimum width=1pt] {};
	\node (P122) at (-3,3) [circle,draw,fill=black!100, inner sep=1pt, minimum width=1pt] {};
	\node (P211) at (-1.8,3) [circle,draw,fill=black!100, inner sep=1pt, minimum width=1pt] {};
	\node (P212) at (-.6,3) [circle,draw,fill=black!100, inner sep=1pt, minimum width=1pt] {};
	\node (P221) at (0.6,3) [circle,draw,fill=black!100, inner sep=1pt, minimum width=1pt] {};
	\node (P222) at (1.8,3) [circle,draw,fill=black!100, inner sep=1pt, minimum width=1pt] {};
	\node (P311) at (3,3) [circle,draw,fill=black!100, inner sep=1pt, minimum width=1pt] {};
	\node (P312) at (4.2,3) [circle,draw,fill=black!100, inner sep=1pt, minimum width=1pt] {};
	\node (P321) at (5.4,3) [circle,draw,fill=black!100, inner sep=1pt, minimum width=1pt] {};
	\node (P322) at (6.6,3) [circle,draw,fill=black!100, inner sep=1pt, minimum width=1pt] {};
	\node (P1111) at (-6.9,4) [circle,draw,fill=black!100, inner sep=1pt, minimum width=1pt] {};
	\node (P1112) at (-6.3,4) [circle,draw,fill=black!100, inner sep=1pt, minimum width=1pt] {};
	\node (P1121) at (-5.7,4) [circle,draw,fill=black!100, inner sep=1pt, minimum width=1pt] {};
	\node (P1122) at (-5.1,4) [circle,draw,fill=black!100, inner sep=1pt, minimum width=1pt] {};
	\node (P1211) at (-4.5,4) [circle,draw,fill=black!100, inner sep=1pt, minimum width=1pt] {};
	\node (P1212) at (-3.9,4) [circle,draw,fill=black!100, inner sep=1pt, minimum width=1pt] {};
	\node (P1221) at (-3.3,4) [circle,draw,fill=black!100, inner sep=1pt, minimum width=1pt] {};
	\node (P1222) at (-2.7,4) [circle,draw,fill=black!100, inner sep=1pt, minimum width=1pt] {};
	\node (P2111) at (-2.1,4) [circle,draw,fill=black!100, inner sep=1pt, minimum width=1pt] {};
	\node (P2112) at (-1.5,4) [circle,draw,fill=black!100, inner sep=1pt, minimum width=1pt] {};
	\node (P2121) at (-.9,4) [circle,draw,fill=black!100, inner sep=1pt, minimum width=1pt] {};
	\node (P2122) at (-.3,4) [circle,draw,fill=black!100, inner sep=1pt, minimum width=1pt] {};
	\node (P2211) at (.3,4) [circle,draw,fill=black!100, inner sep=1pt, minimum width=1pt] {};
	\node (P2212) at (.9,4) [circle,draw,fill=black!100, inner sep=1pt, minimum width=1pt] {};
	\node (P2221) at (1.5,4) [circle,draw,fill=black!100, inner sep=1pt, minimum width=1pt] {};
	\node (P2222) at (2.1,4) [circle,draw,fill=black!100, inner sep=1pt, minimum width=1pt] {};
	\node[label=above:$C_2$] (P3111) at (2.7,4) [circle,draw,fill=black!100, inner sep=1pt, minimum width=1pt] {};
	\node (P3112) at (3.3,4) [circle,draw,fill=black!100, inner sep=1pt, minimum width=1pt] {};
	\node (P3121) at (3.9,4) [circle,draw,fill=black!100, inner sep=1pt, minimum width=1pt] {};
	\node (P3122) at (4.5,4) [circle,draw,fill=black!100, inner sep=1pt, minimum width=1pt] {};
	\node[label=above:$C_3$] (P3211) at (5.1,4) [circle,draw,fill=black!100, inner sep=1pt, minimum width=1pt] {};
	\node (P3212) at (5.7,4) [circle,draw,fill=black!100, inner sep=1pt, minimum width=1pt] {};
	\node (P3221) at (6.3,4) [circle,draw,fill=black!100, inner sep=1pt, minimum width=1pt] {};
	\node (P3222) at (6.9,4) [circle,draw,fill=black!100, inner sep=1pt, minimum width=1pt] {};
	\draw (P0)--(P1)--(P11)--(P111)--(P1111);
	\draw (P111)--(P1112);
	\draw (P11)--(P112)--(P1121);
	\draw (P112)--(P1122);
	\draw (P1)--(P12)--(P121)--(P1211);
	\draw (P121)--(P1212);
	\draw (P12)--(P122)--(P1221);
	\draw (P122)--(P1222);
	\draw (P0)--(P2)--(P21)--(P211)--(P2111);
	\draw (P211)--(P2112);
	\draw (P21)--(P212)--(P2121);
	\draw (P212)--(P2122);
	\draw (P2)--(P22)--(P221)--(P2211);
	\draw (P221)--(P2212);
	\draw (P22)--(P222)--(P2221);
	\draw (P222)--(P2222);
	\draw (P0)--(P3)--(P31)--(P311)--(P3111);
	\draw (P311)--(P3112);
	\draw (P31)--(P312)--(P3121);
	\draw (P312)--(P3122);
	\draw (P3)--(P32)--(P321)--(P3211);
	\draw (P321)--(P3212);
	\draw (P32)--(P322)--(P3221);
	\draw (P322)--(P3222);
	\path[dashed] (P2112.north) edge [out=30,in=150] node [right] {} (P3111.north);
	\path[dashed] (P2212.north) edge [out=30,in=150] node [right] {} (P3211.north);
	\draw (-4.8,1.75) ellipse (45pt and 28pt);
	\end{tikzpicture}
\end{figure}

\begin{proof}[Proof of Lemma~\ref{lem:main}]
The proof of this lemma is by case analysis of the relative ``current" positions of the robber and the cops at any given point in the game when it is the robber's turn to move. We start by assuming that we are playing on a cubic graph of girth at least $8$. We further assume that we are playing with three cops, and that in the ``current" position they have not trapped the robber. We will show that under these assumptions, the cops have also not $2$-trapped the robber unless their configuration is as described in the statement of this lemma. 

Label the edges incident with the robber's current vertex as $e_1$, $e_2$, and $e_3$, and use $v_1$, $v_2$, and $v_3$ to denote the vertices at the opposite end of $e_1$, $e_2$, and $e_3$ (respectively) from the robber's current position. When we speak of a ``branch from $v_i$" ($i \in \{1,2,3\}$) we mean the set of $8$ vertices that include $v_i$ itself, along with all of the vertices that are at distance $4$ or less from the robber along a path that includes a fixed one of the two other neighbours of $v_i$. Figure~\ref{fig:branches} has each of the two branches from $v_2$ circled. Note that by this definition, $v_i$ itself is on both of the branches from $v_i$.

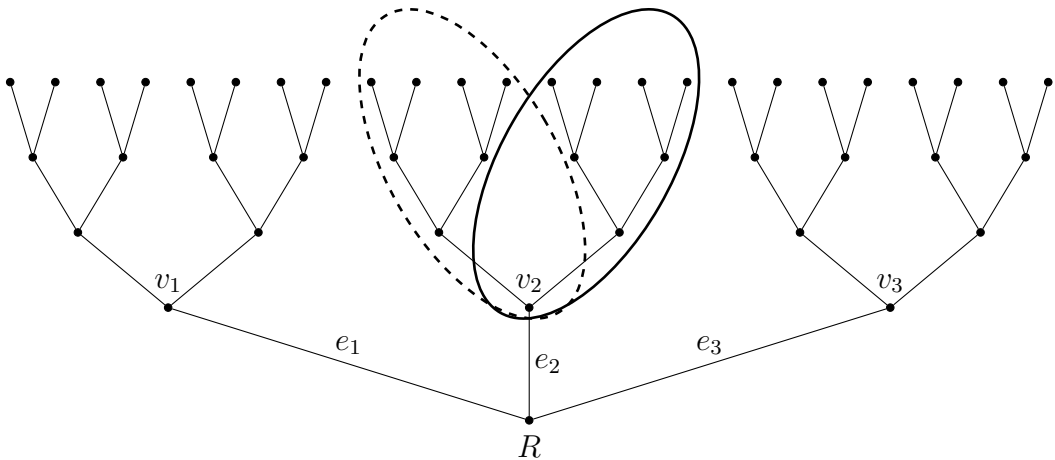
\begin{figure}\caption{Branches from $v_2$ (the left branch is in a dashed ellipse).}\label{fig:branches}
	\begin{tikzpicture}
	\node at (-2.4,.5) {$e_1$};
	\node at (.25,.25) {$e_2$};
	\node at (2.4,.5) {$e_3$};
	\node[label=below:$R$] (P0) at (0,-.5) [circle,draw,fill=black!100, inner sep=1pt, minimum width=1pt] {};
	\node[label=above:$v_1$] (P1) at (-4.8,1) [circle,draw,fill=black!100, inner sep=1pt, minimum width=1pt] {};
	\node[label=above:$v_2$] (P2) at (0,1) [circle,draw,fill=black!100, inner sep=1pt, minimum width=1pt] {};
	\node[label=above:$v_3$] (P3) at (4.8,1) [circle,draw,fill=black!100, inner sep=1pt, minimum width=1pt] {};
	\node (P11) at (-6,2) [circle,draw,fill=black!100, inner sep=1pt, minimum width=1pt] {};
	\node (P12) at (-3.6,2) [circle,draw,fill=black!100, inner sep=1pt, minimum width=1pt] {};
	\node (P21) at (-1.2,2) [circle,draw,fill=black!100, inner sep=1pt, minimum width=1pt] {};
	\node (P22) at (1.2,2) [circle,draw,fill=black!100, inner sep=1pt, minimum width=1pt] {};
	\node (P31) at (3.6,2) [circle,draw,fill=black!100, inner sep=1pt, minimum width=1pt] {};
	\node (P32) at (6,2) [circle,draw,fill=black!100, inner sep=1pt, minimum width=1pt] {};
	\node (P111) at (-6.6,3) [circle,draw,fill=black!100, inner sep=1pt, minimum width=1pt] {};
	\node (P112) at (-5.4,3) [circle,draw,fill=black!100, inner sep=1pt, minimum width=1pt] {};
	\node (P121) at (-4.2,3) [circle,draw,fill=black!100, inner sep=1pt, minimum width=1pt] {};
	\node (P122) at (-3,3) [circle,draw,fill=black!100, inner sep=1pt, minimum width=1pt] {};
	\node (P211) at (-1.8,3) [circle,draw,fill=black!100, inner sep=1pt, minimum width=1pt] {};
	\node (P212) at (-.6,3) [circle,draw,fill=black!100, inner sep=1pt, minimum width=1pt] {};
	\node (P221) at (0.6,3) [circle,draw,fill=black!100, inner sep=1pt, minimum width=1pt] {};
	\node (P222) at (1.8,3) [circle,draw,fill=black!100, inner sep=1pt, minimum width=1pt] {};
	\node (P311) at (3,3) [circle,draw,fill=black!100, inner sep=1pt, minimum width=1pt] {};
	\node (P312) at (4.2,3) [circle,draw,fill=black!100, inner sep=1pt, minimum width=1pt] {};
	\node (P321) at (5.4,3) [circle,draw,fill=black!100, inner sep=1pt, minimum width=1pt] {};
	\node (P322) at (6.6,3) [circle,draw,fill=black!100, inner sep=1pt, minimum width=1pt] {};
	\node (P1111) at (-6.9,4) [circle,draw,fill=black!100, inner sep=1pt, minimum width=1pt] {};
	\node (P1112) at (-6.3,4) [circle,draw,fill=black!100, inner sep=1pt, minimum width=1pt] {};
	\node (P1121) at (-5.7,4) [circle,draw,fill=black!100, inner sep=1pt, minimum width=1pt] {};
	\node (P1122) at (-5.1,4) [circle,draw,fill=black!100, inner sep=1pt, minimum width=1pt] {};
	\node (P1211) at (-4.5,4) [circle,draw,fill=black!100, inner sep=1pt, minimum width=1pt] {};
	\node (P1212) at (-3.9,4) [circle,draw,fill=black!100, inner sep=1pt, minimum width=1pt] {};
	\node (P1221) at (-3.3,4) [circle,draw,fill=black!100, inner sep=1pt, minimum width=1pt] {};
	\node (P1222) at (-2.7,4) [circle,draw,fill=black!100, inner sep=1pt, minimum width=1pt] {};
	\node (P2111) at (-2.1,4) [circle,draw,fill=black!100, inner sep=1pt, minimum width=1pt] {};
	\node (P2112) at (-1.5,4) [circle,draw,fill=black!100, inner sep=1pt, minimum width=1pt] {};
	\node (P2121) at (-.9,4) [circle,draw,fill=black!100, inner sep=1pt, minimum width=1pt] {};
	\node (P2122) at (-.3,4) [circle,draw,fill=black!100, inner sep=1pt, minimum width=1pt] {};
	\node (P2211) at (.3,4) [circle,draw,fill=black!100, inner sep=1pt, minimum width=1pt] {};
	\node (P2212) at (.9,4) [circle,draw,fill=black!100, inner sep=1pt, minimum width=1pt] {};
	\node (P2221) at (1.5,4) [circle,draw,fill=black!100, inner sep=1pt, minimum width=1pt] {};
	\node (P2222) at (2.1,4) [circle,draw,fill=black!100, inner sep=1pt, minimum width=1pt] {};
	\node (P3111) at (2.7,4) [circle,draw,fill=black!100, inner sep=1pt, minimum width=1pt] {};
	\node (P3112) at (3.3,4) [circle,draw,fill=black!100, inner sep=1pt, minimum width=1pt] {};
	\node (P3121) at (3.9,4) [circle,draw,fill=black!100, inner sep=1pt, minimum width=1pt] {};
	\node (P3122) at (4.5,4) [circle,draw,fill=black!100, inner sep=1pt, minimum width=1pt] {};
	\node (P3211) at (5.1,4) [circle,draw,fill=black!100, inner sep=1pt, minimum width=1pt] {};
	\node (P3212) at (5.7,4) [circle,draw,fill=black!100, inner sep=1pt, minimum width=1pt] {};
	\node (P3221) at (6.3,4) [circle,draw,fill=black!100, inner sep=1pt, minimum width=1pt] {};
	\node (P3222) at (6.9,4) [circle,draw,fill=black!100, inner sep=1pt, minimum width=1pt] {};
	\draw (P0)--(P1)--(P11)--(P111)--(P1111);
	\draw (P111)--(P1112);
	\draw (P11)--(P112)--(P1121);
	\draw (P112)--(P1122);
	\draw (P1)--(P12)--(P121)--(P1211);
	\draw (P121)--(P1212);
	\draw (P12)--(P122)--(P1221);
	\draw (P122)--(P1222);
	\draw (P0)--(P2)--(P21)--(P211)--(P2111);
	\draw (P211)--(P2112);
	\draw (P21)--(P212)--(P2121);
	\draw (P212)--(P2122);
	\draw (P2)--(P22)--(P221)--(P2211);
	\draw (P221)--(P2212);
	\draw (P22)--(P222)--(P2221);
	\draw (P222)--(P2222);
	\draw (P0)--(P3)--(P31)--(P311)--(P3111);
	\draw (P311)--(P3112);
	\draw (P31)--(P312)--(P3121);
	\draw (P312)--(P3122);
	\draw (P3)--(P32)--(P321)--(P3211);
	\draw (P321)--(P3212);
	\draw (P32)--(P322)--(P3221);
	\draw (P322)--(P3222);
	\draw[rotate=-30,line width=1pt] (-.8,2.9) ellipse (32pt and 65pt);
	\draw[rotate=30,line width=1pt, dashed] (.8,2.9) ellipse (32pt and 65pt);
	\end{tikzpicture}
	\end{figure}

Recall that since the girth is at least $8$, if a cop is at distance $4$ or more from the robber, more than one of the vertices $v_1$, $v_2$, and $v_3$ may lie on a shortest path between the cop and the robber. Since the cops have not trapped the robber, there is at least one vertex ($v_1$, $v_2$, or $v_3$) that does not lie on a shortest path of length $2$ or less between the robber and some cop. The cases are as follows:
\begin{enumerate}
\item there is no cop within distance $2$ of the robber's current position.
\item there is at least one vertex $v_i$ ($i \in \{1,2,3\}$) that has no cops on at least one of its branches, and any cop on a branch from $v_i$ is not at distance $2$ or less from the robber.
\item every vertex $v_i$ ($i \in \{1,2,3\}$) either has a cop at distance $2$ or less from the robber on one of its branches (and this is the case for some $v_i$), or has at least one cop on each of its branches.
\end{enumerate}

To begin with, let us explain briefly why these cases are exhaustive. If Case~$1$ does not apply, then there is at least one cop within distance $2$ of the robber. Without loss of generality, let us suppose that this cop is on $v_1$ or one of its neighbours. 
If Case~$3$ does not apply, then there is some vertex $v_j$ that has no cop at distance $2$ or less from the robber and no cop on one of its branches. So Case~$2$ applies to $v_j$.

This case analysis will be illustrated with figures. Keep in mind that if the girth is $8$, some of the vertices at distance $4$ from the robber in our figures may be the same as each other, but there are no other duplicated vertices in these illustrations.

\textsc{Case 1. There is no cop within distance $2$ of the robber's current position.}

This situation is illustrated in Figure~\ref{fig:cops-3or4}. This case does cover configurations that are not exactly like the illustration. For example, one cop might be sitting on a vertex that has paths of length $4$ to the robber's position through both $v_1$ and $v_2$, or there might be no cop within distance $4$ on either branch containing $v_1$. However, no cop is on any uncircled vertex of the illustrated subgraph.

\begin{figure}\caption{There is no cop within distance $2$ of the robber's position.}\label{fig:cops-3or4}
\begin{tikzpicture}
\node at (-2.4,.5) {$e_1$};
\node at (.25,.25) {$e_2$};
\node at (2.4,.5) {$e_3$};
\node at (-4.8,3.5) {$C_1$};
\node at (0,3.5) {$C_2$};
\node at (4.8,3.5) {$C_3$};
\node[label=below:$R$] (P0) at (0,-.5) [circle,draw,fill=black!100, inner sep=1pt, minimum width=1pt] {};
\node[label=above:$v_1$] (P1) at (-4.8,1) [circle,draw,fill=black!100, inner sep=1pt, minimum width=1pt] {};
\node[label=above:$v_2$] (P2) at (0,1) [circle,draw,fill=black!100, inner sep=1pt, minimum width=1pt] {};
\node[label=above:$v_3$] (P3) at (4.8,1) [circle,draw,fill=black!100, inner sep=1pt, minimum width=1pt] {};
\node (P11) at (-6,2) [circle,draw,fill=black!100, inner sep=1pt, minimum width=1pt] {};
\node (P12) at (-3.6,2) [circle,draw,fill=black!100, inner sep=1pt, minimum width=1pt] {};
\node (P21) at (-1.2,2) [circle,draw,fill=black!100, inner sep=1pt, minimum width=1pt] {};
\node (P22) at (1.2,2) [circle,draw,fill=black!100, inner sep=1pt, minimum width=1pt] {};
\node (P31) at (3.6,2) [circle,draw,fill=black!100, inner sep=1pt, minimum width=1pt] {};
\node (P32) at (6,2) [circle,draw,fill=black!100, inner sep=1pt, minimum width=1pt] {};
\node (P111) at (-6.6,3) [circle,draw,fill=black!100, inner sep=1pt, minimum width=1pt] {};
\node (P112) at (-5.4,3) [circle,draw,fill=black!100, inner sep=1pt, minimum width=1pt] {};
\node (P121) at (-4.2,3) [circle,draw,fill=black!100, inner sep=1pt, minimum width=1pt] {};
\node (P122) at (-3,3) [circle,draw,fill=black!100, inner sep=1pt, minimum width=1pt] {};
\node (P211) at (-1.8,3) [circle,draw,fill=black!100, inner sep=1pt, minimum width=1pt] {};
\node (P212) at (-.6,3) [circle,draw,fill=black!100, inner sep=1pt, minimum width=1pt] {};
\node (P221) at (0.6,3) [circle,draw,fill=black!100, inner sep=1pt, minimum width=1pt] {};
\node (P222) at (1.8,3) [circle,draw,fill=black!100, inner sep=1pt, minimum width=1pt] {};
\node (P311) at (3,3) [circle,draw,fill=black!100, inner sep=1pt, minimum width=1pt] {};
\node (P312) at (4.2,3) [circle,draw,fill=black!100, inner sep=1pt, minimum width=1pt] {};
\node (P321) at (5.4,3) [circle,draw,fill=black!100, inner sep=1pt, minimum width=1pt] {};
\node (P322) at (6.6,3) [circle,draw,fill=black!100, inner sep=1pt, minimum width=1pt] {};
\node (P1111) at (-6.9,4) [circle,draw,fill=black!100, inner sep=1pt, minimum width=1pt] {};
\node (P1112) at (-6.3,4) [circle,draw,fill=black!100, inner sep=1pt, minimum width=1pt] {};
\node (P1121) at (-5.7,4) [circle,draw,fill=black!100, inner sep=1pt, minimum width=1pt] {};
\node (P1122) at (-5.1,4) [circle,draw,fill=black!100, inner sep=1pt, minimum width=1pt] {};
\node (P1211) at (-4.5,4) [circle,draw,fill=black!100, inner sep=1pt, minimum width=1pt] {};
\node (P1212) at (-3.9,4) [circle,draw,fill=black!100, inner sep=1pt, minimum width=1pt] {};
\node (P1221) at (-3.3,4) [circle,draw,fill=black!100, inner sep=1pt, minimum width=1pt] {};
\node (P1222) at (-2.7,4) [circle,draw,fill=black!100, inner sep=1pt, minimum width=1pt] {};
\node (P2111) at (-2.1,4) [circle,draw,fill=black!100, inner sep=1pt, minimum width=1pt] {};
\node (P2112) at (-1.5,4) [circle,draw,fill=black!100, inner sep=1pt, minimum width=1pt] {};
\node (P2121) at (-.9,4) [circle,draw,fill=black!100, inner sep=1pt, minimum width=1pt] {};
\node (P2122) at (-.3,4) [circle,draw,fill=black!100, inner sep=1pt, minimum width=1pt] {};
\node (P2211) at (.3,4) [circle,draw,fill=black!100, inner sep=1pt, minimum width=1pt] {};
\node (P2212) at (.9,4) [circle,draw,fill=black!100, inner sep=1pt, minimum width=1pt] {};
\node (P2221) at (1.5,4) [circle,draw,fill=black!100, inner sep=1pt, minimum width=1pt] {};
\node (P2222) at (2.1,4) [circle,draw,fill=black!100, inner sep=1pt, minimum width=1pt] {};
\node (P3111) at (2.7,4) [circle,draw,fill=black!100, inner sep=1pt, minimum width=1pt] {};
\node (P3112) at (3.3,4) [circle,draw,fill=black!100, inner sep=1pt, minimum width=1pt] {};
\node (P3121) at (3.9,4) [circle,draw,fill=black!100, inner sep=1pt, minimum width=1pt] {};
\node (P3122) at (4.5,4) [circle,draw,fill=black!100, inner sep=1pt, minimum width=1pt] {};
\node (P3211) at (5.1,4) [circle,draw,fill=black!100, inner sep=1pt, minimum width=1pt] {};
\node (P3212) at (5.7,4) [circle,draw,fill=black!100, inner sep=1pt, minimum width=1pt] {};
\node (P3221) at (6.3,4) [circle,draw,fill=black!100, inner sep=1pt, minimum width=1pt] {};
\node (P3222) at (6.9,4) [circle,draw,fill=black!100, inner sep=1pt, minimum width=1pt] {};
\draw (P0)--(P1)--(P11)--(P111)--(P1111);
\draw (P111)--(P1112);
\draw (P11)--(P112)--(P1121);
\draw (P112)--(P1122);
\draw (P1)--(P12)--(P121)--(P1211);
\draw (P121)--(P1212);
\draw (P12)--(P122)--(P1221);
\draw (P122)--(P1222);
\draw (P0)--(P2)--(P21)--(P211)--(P2111);
\draw (P211)--(P2112);
\draw (P21)--(P212)--(P2121);
\draw (P212)--(P2122);
\draw (P2)--(P22)--(P221)--(P2211);
\draw (P221)--(P2212);
\draw (P22)--(P222)--(P2221);
\draw (P222)--(P2222);
\draw (P0)--(P3)--(P31)--(P311)--(P3111);
\draw (P311)--(P3112);
\draw (P31)--(P312)--(P3121);
\draw (P312)--(P3122);
\draw (P3)--(P32)--(P321)--(P3211);
\draw (P321)--(P3212);
\draw (P32)--(P322)--(P3221);
\draw (P322)--(P3222);
\draw (-4.8,3.5) ellipse (67pt and 40pt);
\draw (0,3.5) ellipse (67pt and 40pt);
\draw (4.8,3.5) ellipse (67pt and 40pt);
\end{tikzpicture}
\end{figure}
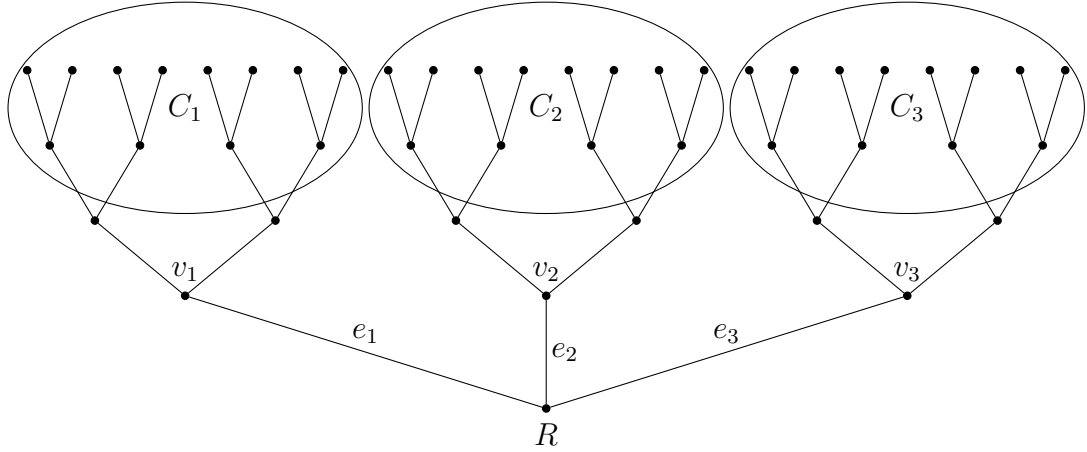

In Case~$1$, the robber may move to any of its adjacent vertices ($v_1$, $v_2$, or $v_3$). Without loss of generality, suppose that the robber chooses to move to vertex $v_1$. The cops' initial positions mean that none of them can use their move to reach the robber's initial vertex, or vertices $v_2$ or $v_3$. Therefore, the cops have not trapped the robber on $v_1$. This means that the cops had not $2$-trapped the robber previously.

\textsc{Case 2. There is at least one vertex $v_i$ ($i \in \{1,2,3\}$) that has no cops on at least one of its branches, and any cop on a branch from $v_i$ is not at distance $2$ or less from the robber.}

In Figure~\ref{fig:cop-far} we have illustrated this case. Without loss of generality, suppose (as in the illustration) that vertex $v_2$ satisfies the hypothesis. Since there is no cop at distance $2$ or less  there is no cop on $v_2$ or its neighbours; since there is no cop on one the branches from $v_2$, we may assume without loss of generality that if any cops are here at all, they are on the left-hand branch. 

\begin{figure}\caption{There is no cop on one of the branches from $v_2$, and any cops on a branch from $v_2$ are not at distance $2$ or less (from the robber). So (assuming any cops are on the left branch) there are no cops in encircled areas.}\label{fig:cop-far}
\begin{tikzpicture}
\node at (-2.4,.5) {$e_1$};
\node at (.25,.25) {$e_2$};
\node at (2.4,.5) {$e_3$};
\node[label=below:$R$] (P0) at (0,-.5) [circle,draw,fill=black!100, inner sep=1pt, minimum width=1pt] {};
\node[label=above:$v_1$] (P1) at (-4.8,1) [circle,draw,fill=black!100, inner sep=1pt, minimum width=1pt] {};
\node[label=above:$v_2$] (P2) at (0,1) [circle,draw,fill=black!100, inner sep=1pt, minimum width=1pt] {};
\node[label=above:$v_3$] (P3) at (4.8,1) [circle,draw,fill=black!100, inner sep=1pt, minimum width=1pt] {};
\node (P11) at (-6,2) [circle,draw,fill=black!100, inner sep=1pt, minimum width=1pt] {};
\node (P12) at (-3.6,2) [circle,draw,fill=black!100, inner sep=1pt, minimum width=1pt] {};
\node (P21) at (-1.2,2) [circle,draw,fill=black!100, inner sep=1pt, minimum width=1pt] {};
\node (P22) at (1.2,2) [circle,draw,fill=black!100, inner sep=1pt, minimum width=1pt] {};
\node (P31) at (3.6,2) [circle,draw,fill=black!100, inner sep=1pt, minimum width=1pt] {};
\node (P32) at (6,2) [circle,draw,fill=black!100, inner sep=1pt, minimum width=1pt] {};
\node (P111) at (-6.6,3) [circle,draw,fill=black!100, inner sep=1pt, minimum width=1pt] {};
\node (P112) at (-5.4,3) [circle,draw,fill=black!100, inner sep=1pt, minimum width=1pt] {};
\node (P121) at (-4.2,3) [circle,draw,fill=black!100, inner sep=1pt, minimum width=1pt] {};
\node (P122) at (-3,3) [circle,draw,fill=black!100, inner sep=1pt, minimum width=1pt] {};
\node (P211) at (-1.8,3) [circle,draw,fill=black!100, inner sep=1pt, minimum width=1pt] {};
\node (P212) at (-.6,3) [circle,draw,fill=black!100, inner sep=1pt, minimum width=1pt] {};
\node (P221) at (0.6,3) [circle,draw,fill=black!100, inner sep=1pt, minimum width=1pt] {};
\node (P222) at (1.8,3) [circle,draw,fill=black!100, inner sep=1pt, minimum width=1pt] {};
\node (P311) at (3,3) [circle,draw,fill=black!100, inner sep=1pt, minimum width=1pt] {};
\node (P312) at (4.2,3) [circle,draw,fill=black!100, inner sep=1pt, minimum width=1pt] {};
\node (P321) at (5.4,3) [circle,draw,fill=black!100, inner sep=1pt, minimum width=1pt] {};
\node (P322) at (6.6,3) [circle,draw,fill=black!100, inner sep=1pt, minimum width=1pt] {};
\node (P1111) at (-6.9,4) [circle,draw,fill=black!100, inner sep=1pt, minimum width=1pt] {};
\node (P1112) at (-6.3,4) [circle,draw,fill=black!100, inner sep=1pt, minimum width=1pt] {};
\node (P1121) at (-5.7,4) [circle,draw,fill=black!100, inner sep=1pt, minimum width=1pt] {};
\node (P1122) at (-5.1,4) [circle,draw,fill=black!100, inner sep=1pt, minimum width=1pt] {};
\node (P1211) at (-4.5,4) [circle,draw,fill=black!100, inner sep=1pt, minimum width=1pt] {};
\node (P1212) at (-3.9,4) [circle,draw,fill=black!100, inner sep=1pt, minimum width=1pt] {};
\node (P1221) at (-3.3,4) [circle,draw,fill=black!100, inner sep=1pt, minimum width=1pt] {};
\node (P1222) at (-2.7,4) [circle,draw,fill=black!100, inner sep=1pt, minimum width=1pt] {};
\node (P2111) at (-2.1,4) [circle,draw,fill=black!100, inner sep=1pt, minimum width=1pt] {};
\node (P2112) at (-1.5,4) [circle,draw,fill=black!100, inner sep=1pt, minimum width=1pt] {};
\node (P2121) at (-.9,4) [circle,draw,fill=black!100, inner sep=1pt, minimum width=1pt] {};
\node (P2122) at (-.3,4) [circle,draw,fill=black!100, inner sep=1pt, minimum width=1pt] {};
\node (P2211) at (.3,4) [circle,draw,fill=black!100, inner sep=1pt, minimum width=1pt] {};
\node (P2212) at (.9,4) [circle,draw,fill=black!100, inner sep=1pt, minimum width=1pt] {};
\node (P2221) at (1.5,4) [circle,draw,fill=black!100, inner sep=1pt, minimum width=1pt] {};
\node (P2222) at (2.1,4) [circle,draw,fill=black!100, inner sep=1pt, minimum width=1pt] {};
\node (P3111) at (2.7,4) [circle,draw,fill=black!100, inner sep=1pt, minimum width=1pt] {};
\node (P3112) at (3.3,4) [circle,draw,fill=black!100, inner sep=1pt, minimum width=1pt] {};
\node (P3121) at (3.9,4) [circle,draw,fill=black!100, inner sep=1pt, minimum width=1pt] {};
\node (P3122) at (4.5,4) [circle,draw,fill=black!100, inner sep=1pt, minimum width=1pt] {};
\node (P3211) at (5.1,4) [circle,draw,fill=black!100, inner sep=1pt, minimum width=1pt] {};
\node (P3212) at (5.7,4) [circle,draw,fill=black!100, inner sep=1pt, minimum width=1pt] {};
\node (P3221) at (6.3,4) [circle,draw,fill=black!100, inner sep=1pt, minimum width=1pt] {};
\node (P3222) at (6.9,4) [circle,draw,fill=black!100, inner sep=1pt, minimum width=1pt] {};
\draw (P0)--(P1)--(P11)--(P111)--(P1111);
\draw (P111)--(P1112);
\draw (P11)--(P112)--(P1121);
\draw (P112)--(P1122);
\draw (P1)--(P12)--(P121)--(P1211);
\draw (P121)--(P1212);
\draw (P12)--(P122)--(P1221);
\draw (P122)--(P1222);
\draw (P0)--(P2)--(P21)--(P211)--(P2111);
\draw (P211)--(P2112);
\draw (P21)--(P212)--(P2121);
\draw (P212)--(P2122);
\draw (P2)--(P22)--(P221)--(P2211);
\draw (P221)--(P2212);
\draw (P22)--(P222)--(P2221);
\draw (P222)--(P2222);
\draw (P0)--(P3)--(P31)--(P311)--(P3111);
\draw (P311)--(P3112);
\draw (P31)--(P312)--(P3121);
\draw (P312)--(P3122);
\draw (P3)--(P32)--(P321)--(P3211);
\draw (P321)--(P3212);
\draw (P32)--(P322)--(P3221);
\draw (P322)--(P3222);
\draw (1.2,3.55) ellipse (35pt and 25pt);
\draw (0,1.75) ellipse (40pt and 26pt);
\end{tikzpicture}
\end{figure}
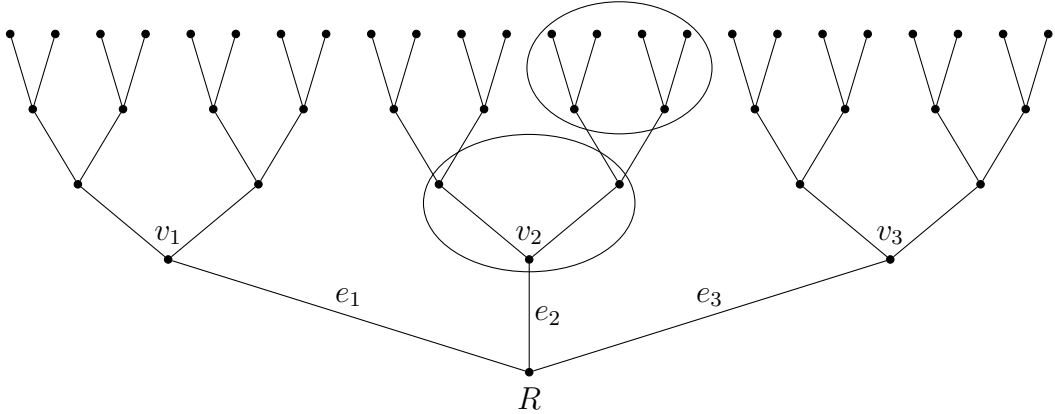

In this case, the robber should move to vertex $v_2$. Since no cop was within $2$ of the robber on either branch from $v_2$, no cop is in a position to capture the robber immediately on $v_2$. 
Since no cop was on any vertex of the right-hand branch from $v_2$, after the cops' move no cop can be within distance $2$ of $v_2$ on this branch. Therefore the cops will not have trapped the robber after their move.

Once again, this means that the cops had not $2$-trapped the robber in the initial configuration.

\textsc{Case 3. Every vertex $v_i$ ($i \in \{1,2,3\}$) either has a cop at distance $2$ or less from the robber on one of its branches (and this is the case for some $v_i$), or has at least one cop on each of its branches.}

Without loss of generality, we may assume that cop $C_1$ is at distance $2$ or less from the robber, being on vertex $v_1$ or one of its neighbours (other than the robber's vertex). 

Suppose momentarily that there is a second cop at distance $2$ or less from the robber on one of the branches from one of the other neighbours of the robber's vertex. Without loss of generality, we may assume that this neighbour is $v_2$. Since neither of these cops is on a branch of $v_3$, there cannot be a cop on each of the branches from $v_3$ unless that cop is also at distance $2$ or less from the robber. This contradicts our hypothesis that the robber is not yet trapped. 

Thus, there must be one cop on each of the branches from $v_2$ (and neither of these cops is within distance $2$ of the robber) and one cop on each of the branches from $v_3$ (not within distance $2$ of the robber). Due to our girth hypothesis, a single cop cannot be on both branches from a vertex $v_i$ unless she is on vertex $v_i$ (and thus within $2$ of the robber). Since $C_1$ is not on either branch from $v_2$ or either branch from $v_3$, these four branches must be covered by cops $C_2$ and $C_3$. The only way a cop can be on a branch from $v_2$ and a branch from $v_3$ is for it to be sitting at distance $4$ from the robber, on a vertex that is simultaneously on a branch from $v_2$ and a branch from $v_3$.

This means that each of $C_2$ and $C_3$ is at the antipodal vertex of a cycle of length $8$ from the robber, and (since all four branches from $v_2$ and $v_3$ are involved in these $8$-cycles) the intersection of these two cycles consists precisely of the edges $e_2$ and $e_3$. This is the configuration described in the statement of our lemma as $2$-trapping the robber, and illustrated in Figure~\ref{fig:2-trapped}.
\end{proof}

This lemma allows us to prove our main result.

\begin{thm}\label{thm:main}
Let $G$ be a cubic graph of girth at least $8$. Unless $G$ contains two cycles of length $8$ whose intersection is a path of length $2$ (with two edges and three vertices), we conclude that $c(G) \ge 4$: in particular, if $G$ is a generalized Petersen graph, then this means $c(G)=4$.
\end{thm}

\begin{proof}
Suppose that we play the game on $G$ with $3$ cops.

Choose any initial locations for the cops, and consider the set of vertices that are at distance $3$ from $C_1$. Since the girth of $G$ is at least $8$ and the valency is $3$, there are $12$ such vertices, all distinct. At most two of these vertices can hold the other two cops. Therefore, there are at least $10$ vertices the robber can choose for his initial position that are not occupied by any of the cops, and are not within distance $2$ of cop $C_1$. Therefore the robber has choices for his initial move that do not leave him trapped by the cops. 

Using Lemma~\ref{lem:main} inductively, we see that since the robber has not been trapped by the cops, he can always avoid becoming trapped, and can therefore win the game. (Our hypothesis about the structure of $G$ together with Lemma~\ref{lem:main} implies that the robber can never be $2$-trapped, and will therefore never become trapped.) 	

Therefore, $c(G)>3$, so $c(G) \ge 4$. 

By \cite{ball}, if $G$ is a generalized Petersen graph, then $c(G)\le 4$, so $c(G)=4$.
\end{proof}

Notice that this result does not necessarily imply that cubic graphs of girth $8$ that do have cycles of length $8$ whose intersection is a path of length $2$ do actually have cop number $3$ or less. All we have shown so far about such graphs is that they do admit placements for $3$ cops and a robber in which the cops have $2$-trapped the robber. 

Just as three cops could never actually force the robber into a trapped position in other cubic graphs of girth $8$, it is possible that three cops cannot actually force the robber into a $2$-trapped position in some (or all) families of cubic graphs in which a $2$-trapped position exists. To determine this will require further understanding of the structure of these graphs. 

In the next section, we undertake an analysis of which generalized Petersen graphs of girth $8$ admit a $2$-trapped configuration. Before doing so, we provide a generalization of Theorem~\ref{thm:main} that may be of broader interest but uses very similar arguments. It is possible that this result is known, but we did not find it in our research into the literature on this problem. The most closely-related result seems to be Frankl's bound \cite{frankl}, showing that if $\delta(G)>d$ and the girth of $G$ is at least $8t-3$ then $c(G)>d^t$. However, for girth $9$ we would only have $t=1$, so Frankl's bound would only imply that $c(G)\ge \delta(G)$, not (as we will show) that $c(G)>\delta(G)$.

\begin{thm}
If $G$ is a connected graph of minimum valency $\delta\ge 3$ and girth at least $9$, then $c(G)>\delta$.
\end{thm}

\begin{proof}
We want to show that $\delta$ cops cannot capture the robber with optimal play on both sides, because the robber can always avoid becoming trapped. So assume that we are playing the game with $\delta$ cops.

As in the proof of Theorem~\ref{thm:main}, since every vertex has valency at least $\delta$ and the girth is at least $9$, there are at least $\delta(\delta-1)^2 \ge 4\delta$ vertices at distance $3$ from the initial position of cop $C_1$, and the robber can choose any one of these that is not occupied by one of the other $\delta-1$ cops. So the robber has numerous options for an initial position that is not trapped, since it is not within distance $2$ of cop $C_1$. (Every one of the $\delta$ or more neighbours of $v$ must have a cop on either itself or one of its neighbours in order to trap the robber, which requires all $\delta$ cops, so $C_1$ needs to be involved in any trapping of the robber.)

Because we are assuming the girth is at least $9$ (rather than $8$), we can employ a simplified version of the proof of Lemma~\ref{lem:main}. Inductively assume that the robber has not yet been trapped, and call the robber's current vertex $v$. Then there is some neighbour $u$ of $v$ such that no cop is on $u$ or any of the neighbours of $u$. If it is possible to choose $u$ so that no cop is at distance $4$ or less from $v$ on a path that passes through $u$, then we do so.

With the goal of contradicting our choice of $u$, suppose that there is more than one cop at distance $4$ or less from $v$ whose shortest path uses $u$. Since there are at least $\delta$ neighbours of $v$ and only $\delta$ cops, and because the girth of $G$ is  at least $9$, the fact that the shortest paths from two of the cops to $v$ have distance $4$ or less and pass through $u$ means that there must be some other neighbour $x$ of $v$ such that none of the shortest paths from a cop to $v$ that have distance $4$ or less pass through $x$. But in this circumstance, our criteria for choosing should have led us to choose $x$ rather than $u$, a contradiction.

So we may assume that at most one cop is at distance $4$ or less from $v$ along a shortest path that uses $u$. The robber should move to $u$. Since there was no cop on $u$ or any of its neighbours, the cops cannot immediately capture the robber. 

Since $u$ has at least two neighbours other than $v$ ($\delta \ge 3$), one of these neighbours (say $y$) is not used in the shortest path from any cop who was within $4$ of $v$, to $v$. 
But this means that after the robber has moved to $u$, there will be no cop on $y$ or any of its neighbours, so the robber is not trapped.

Thus, the robber always has a move that does not allow the cops to trap him if he was not already trapped, so the cops cannot win.
\end{proof}

\section{Graphs of girth $8$ on which the robber can be $2$-trapped}\label{sec:numco}

In this section, we determine for which values of $n$ and $k$ the graph $GP(n,k)$ admits two cycles of length 8 whose intersection is a path of length 2. In other words, we are finding all the families of generalized Petersen graphs that admit a configuration in which a robber has been 2-trapped by three cops. 

We first label the vertices in the sub-tree arbitrarily. Assuming the path of length 2 has a fixed vertex $v$ as its central vertex, we see that we  have two possible  labellings, as $v$ can be in either $A$ (shown in Figure~\ref{fig:labelA}) or in $B$ (shown in Figure~\ref{fig:labelB}). 

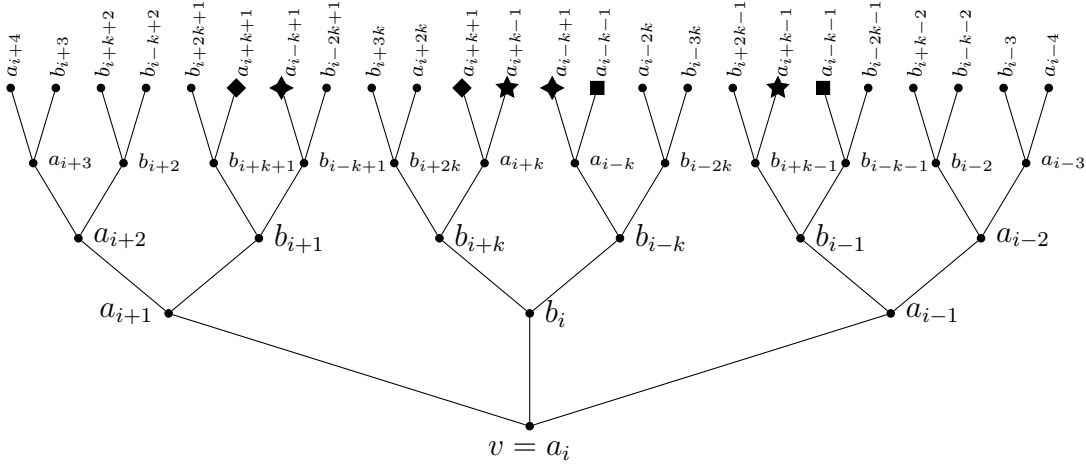
\begin{figure}\caption{If $v$ is in $A$, the labelling is as follows (where $v$ is now labelled $a_i$).}\label{fig:labelA}
\begin{tikzpicture}
\node[label=below:{$v=a_i$}] (P0) at (0,-.5) [circle,draw,fill=black!100, inner sep=1pt, minimum width=1pt] {};
\node[label=left:$a_{i+1}$] (P1) at (-4.8,1) [circle,draw,fill=black!100, inner sep=1pt, minimum width=1pt] {};
\node[label=right:$b_{i}$] (P2) at (0,1) [circle,draw,fill=black!100, inner sep=1pt, minimum width=1pt] {};
\node[label=right:$a_{i-1}$] (P3) at (4.8,1) [circle,draw,fill=black!100, inner sep=1pt, minimum width=1pt] {};
\node[label=right:$a_{i+2}$] (P11) at (-6,2) [circle,draw,fill=black!100, inner sep=1pt, minimum width=1pt] {};
\node[label=right:$b_{i+1}$] (P12) at (-3.6,2) [circle,draw,fill=black!100, inner sep=1pt, minimum width=1pt] {};
\node[label=right:$b_{i+k}$] (P21) at (-1.2,2) [circle,draw,fill=black!100, inner sep=1pt, minimum width=1pt] {};
\node[label=right:$b_{i-k}$] (P22) at (1.2,2) [circle,draw,fill=black!100, inner sep=1pt, minimum width=1pt] {};
\node[label=right:$b_{i-1}$] (P31) at (3.6,2) [circle,draw,fill=black!100, inner sep=1pt, minimum width=1pt] {};
\node[label=right:$a_{i-2}$] (P32) at (6,2) [circle,draw,fill=black!100, inner sep=1pt, minimum width=1pt] {};
\node[label={[font=\tiny]right:$a_{i+3}$}] (P111) at (-6.6,3) [circle,draw,fill=black!100, inner sep=1pt, minimum width=1pt] {};
\node[label={[font=\tiny]right:$b_{i+2}$}] (P112) at (-5.4,3) [circle,draw,fill=black!100, inner sep=1pt, minimum width=1pt] {};
\node[label={[font=\tiny]right:$b_{i+k+1}$}] (P121) at (-4.2,3) [circle,draw,fill=black!100, inner sep=1pt, minimum width=1pt] {};
\node[label={[font=\tiny]right:$b_{i-k+1}$}] (P122) at (-3,3) [circle,draw,fill=black!100, inner sep=1pt, minimum width=1pt] {};
\node[label={[font=\tiny]right:$b_{i+2k}$}] (P211) at (-1.8,3) [circle,draw,fill=black!100, inner sep=1pt, minimum width=1pt] {};
\node[label={[font=\tiny]right:$a_{i+k}$}] (P212) at (-.6,3) [circle,draw,fill=black!100, inner sep=1pt, minimum width=1pt] {};
\node[label={[font=\tiny]right:$a_{i-k}$}] (P221) at (0.6,3) [circle,draw,fill=black!100, inner sep=1pt, minimum width=1pt] {};
\node[label={[font=\tiny]right:$b_{i-2k}$}] (P222) at (1.8,3) [circle,draw,fill=black!100, inner sep=1pt, minimum width=1pt] {};
\node[label={[font=\tiny]right:$b_{i+k-1}$}] (P311) at (3,3) [circle,draw,fill=black!100, inner sep=1pt, minimum width=1pt] {};
\node[label={[font=\tiny]right:$b_{i-k-1}$}] (P312) at (4.2,3) [circle,draw,fill=black!100, inner sep=1pt, minimum width=1pt] {};
\node[label={[font=\tiny]right:$b_{i-2}$}] (P321) at (5.4,3) [circle,draw,fill=black!100, inner sep=1pt, minimum width=1pt] {};
\node[label={[font=\tiny]right:$a_{i-3}$}] (P322) at (6.6,3) [circle,draw,fill=black!100, inner sep=1pt, minimum width=1pt] {};
\node[label={[font=\tiny,rotate=90]right:$a_{i+4}$}] (P1111) at (-6.9,4) [circle,draw,fill=black!100, inner sep=1pt, minimum width=1pt] {};
\node[label={[font=\tiny,rotate=90]right:$b_{i+3}$}] (P1112) at (-6.3,4) [circle,draw,fill=black!100, inner sep=1pt, minimum width=1pt] {};
\node[label={[font=\tiny,rotate=90]right:$b_{i+k+2}$}] (P1121) at (-5.7,4) [circle,draw,fill=black!100, inner sep=1pt, minimum width=1pt] {};
\node[label={[font=\tiny,rotate=90]right:$b_{i-k+2}$}] (P1122) at (-5.1,4) [circle,draw,fill=black!100, inner sep=1pt, minimum width=1pt] {};
\node[label={[font=\tiny,rotate=90]right:$b_{i+2k+1}$}] (P1211) at (-4.5,4) [circle,draw,fill=black!100, inner sep=1pt, minimum width=1pt] {};
\node[label={[font=\tiny,rotate=90]right:$a_{i+k+1}$}] (P1212) at (-3.9,4) [diamond,draw,fill=black!100, inner sep=1.75pt, minimum width=1.75pt] {};
\node[label={[font=\tiny,rotate=90]right:$a_{i-k+1}$}] (P1221) at (-3.3,4) [star,star points=4,star point ratio=2,draw,fill=black!100, inner sep=1.5pt, minimum width=1.5pt] {};
\node[label={[font=\tiny,rotate=90]right:$b_{i-2k+1}$}] (P1222) at (-2.7,4) [circle,draw,fill=black!100, inner sep=1pt, minimum width=1pt] {};
\node[label={[font=\tiny,rotate=90]right:$b_{i+3k}$}] (P2111) at (-2.1,4) [circle,draw,fill=black!100, inner sep=1pt, minimum width=1pt] {};
\node[label={[font=\tiny,rotate=90]right:$a_{i+2k}$}] (P2112) at (-1.5,4) [circle,draw,fill=black!100, inner sep=1pt, minimum width=1pt] {};
\node[label={[font=\tiny,rotate=90]right:$a_{i+k+1}$}] (P2121) at (-.9,4) [diamond,draw,fill=black!100, inner sep=1.75pt, minimum width=1.75pt] {};
\node[label={[font=\tiny,rotate=90]right:$a_{i+k-1}$}] (P2122) at (-.3,4) [star,star points=5,star point ratio=2,draw,fill=black!100, inner sep=1.5pt, minimum width=1.5pt] {};
\node[label={[font=\tiny,rotate=90]right:$a_{i-k+1}$}] (P2211) at (.3,4) [star,star points=4,star point ratio=2,draw,fill=black!100, inner sep=1.5pt, minimum width=1.5pt] {};
\node[label={[font=\tiny,rotate=90]right:$a_{i-k-1}$}] (P2212) at (.9,4) [regular polygon,regular polygon sides=4,draw,fill=black!100, inner sep=1.75pt, minimum width=1.75pt] {};
\node[label={[font=\tiny,rotate=90]right:$a_{i-2k}$}] (P2221) at (1.5,4) [circle,draw,fill=black!100, inner sep=1pt, minimum width=1pt] {};
\node[label={[font=\tiny,rotate=90]right:$b_{i-3k}$}] (P2222) at (2.1,4) [circle,draw,fill=black!100, inner sep=1pt, minimum width=1pt] {};
\node[label={[font=\tiny,rotate=90]right:$b_{i+2k-1}$}] (P3111) at (2.7,4) [circle,draw,fill=black!100, inner sep=1pt, minimum width=1pt] {};
\node[label={[font=\tiny,rotate=90]right:$a_{i+k-1}$}] (P3112) at (3.3,4) [star,star points=5,star point ratio=2,draw,fill=black!100, inner sep=1.5pt, minimum width=1.5pt] {};
\node[label={[font=\tiny,rotate=90]right:$a_{i-k-1}$}] (P3121) at (3.9,4) [regular polygon,regular polygon sides=4,draw,fill=black!100, inner sep=1.75pt, minimum width=1.75pt] {};
\node[label={[font=\tiny,rotate=90]right:$b_{i-2k-1}$}] (P3122) at (4.5,4) [circle,draw,fill=black!100, inner sep=1pt, minimum width=1pt] {};
\node[label={[font=\tiny,rotate=90]right:$b_{i+k-2}$}] (P3211) at (5.1,4) [circle,draw,fill=black!100, inner sep=1pt, minimum width=1pt] {};
\node[label={[font=\tiny,rotate=90]right:$b_{i-k-2}$}] (P3212) at (5.7,4) [circle,draw,fill=black!100, inner sep=1pt, minimum width=1pt] {};
\node[label={[font=\tiny,rotate=90]right:$b_{i-3}$}] (P3221) at (6.3,4) [circle,draw,fill=black!100, inner sep=1pt, minimum width=1pt] {};
\node[label={[font=\tiny,rotate=90]right:$a_{i-4}$}] (P3222) at (6.9,4) [circle,draw,fill=black!100, inner sep=1pt, minimum width=1pt] {};
\draw (P0)--(P1)--(P11)--(P111)--(P1111);
\draw (P111)--(P1112);
\draw (P11)--(P112)--(P1121);
\draw (P112)--(P1122);
\draw (P1)--(P12)--(P121)--(P1211);
\draw (P121)--(P1212);
\draw (P12)--(P122)--(P1221);
\draw (P122)--(P1222);
\draw (P0)--(P2)--(P21)--(P211)--(P2111);
\draw (P211)--(P2112);
\draw (P21)--(P212)--(P2121);
\draw (P212)--(P2122);
\draw (P2)--(P22)--(P221)--(P2211);
\draw (P221)--(P2212);
\draw (P22)--(P222)--(P2221);
\draw (P222)--(P2222);
\draw (P0)--(P3)--(P31)--(P311)--(P3111);
\draw (P311)--(P3112);
\draw (P31)--(P312)--(P3121);
\draw (P312)--(P3122);
\draw (P3)--(P32)--(P321)--(P3211);
\draw (P321)--(P3212);
\draw (P32)--(P322)--(P3221);
\draw (P322)--(P3222);
\end{tikzpicture}
\end{figure}

\begin{figure}\caption{If $v$ is in $B$, the labelling is as follows (where $v$ is now labelled $b_i$).}\label{fig:labelB}
\begin{tikzpicture}
\node[label=below:{$v=b_i$}] (P0) at (0,-.5) [circle,draw,fill=black!100, inner sep=1pt, minimum width=1pt] {};
\node[label=left:$b_{i+k}$] (P1) at (-4.8,1) [circle,draw,fill=black!100, inner sep=1pt, minimum width=1pt] {};
\node[label=right:$a_{i}$] (P2) at (0,1) [circle,draw,fill=black!100, inner sep=1pt, minimum width=1pt] {};
\node[label=right:$b_{i-k}$] (P3) at (4.8,1) [circle,draw,fill=black!100, inner sep=1pt, minimum width=1pt] {};
\node[label=right:$b_{i+2k}$] (P11) at (-6,2) [circle,draw,fill=black!100, inner sep=1pt, minimum width=1pt] {};
\node[label=right:$a_{i+k}$] (P12) at (-3.6,2) [circle,draw,fill=black!100, inner sep=1pt, minimum width=1pt] {};
\node[label=right:$a_{i+1}$] (P21) at (-1.2,2) [circle,draw,fill=black!100, inner sep=1pt, minimum width=1pt] {};
\node[label=right:$a_{i-1}$] (P22) at (1.2,2) [circle,draw,fill=black!100, inner sep=1pt, minimum width=1pt] {};
\node[label=right:$a_{i-k}$] (P31) at (3.6,2) [circle,draw,fill=black!100, inner sep=1pt, minimum width=1pt] {};
\node[label=right:$b_{i-2k}$] (P32) at (6,2) [circle,draw,fill=black!100, inner sep=1pt, minimum width=1pt] {};
\node[label={[font=\tiny]right:$b_{i+3k}$}] (P111) at (-6.6,3) [circle,draw,fill=black!100, inner sep=1pt, minimum width=1pt] {};
\node[label={[font=\tiny]right:$a_{i+2k}$}] (P112) at (-5.4,3) [circle,draw,fill=black!100, inner sep=1pt, minimum width=1pt] {};
\node[label={[font=\tiny]right:$a_{i+k+1}$}] (P121) at (-4.2,3) [circle,draw,fill=black!100, inner sep=1pt, minimum width=1pt] {};
\node[label={[font=\tiny]right:$a_{i+k-1}$}] (P122) at (-3,3) [circle,draw,fill=black!100, inner sep=1pt, minimum width=1pt] {};
\node[label={[font=\tiny]right:$a_{i+2}$}] (P211) at (-1.8,3) [circle,draw,fill=black!100, inner sep=1pt, minimum width=1pt] {};
\node[label={[font=\tiny]right:$b_{i+1}$}] (P212) at (-.6,3) [circle,draw,fill=black!100, inner sep=1pt, minimum width=1pt] {};
\node[label={[font=\tiny]right:$b_{i-1}$}] (P221) at (0.6,3) [circle,draw,fill=black!100, inner sep=1pt, minimum width=1pt] {};
\node[label={[font=\tiny]right:$a_{i-2}$}] (P222) at (1.8,3) [circle,draw,fill=black!100, inner sep=1pt, minimum width=1pt] {};
\node[label={[font=\tiny]right:$a_{i-k+1}$}] (P311) at (3,3) [circle,draw,fill=black!100, inner sep=1pt, minimum width=1pt] {};
\node[label={[font=\tiny]right:$a_{i-k-1}$}] (P312) at (4.2,3) [circle,draw,fill=black!100, inner sep=1pt, minimum width=1pt] {};
\node[label={[font=\tiny]right:$a_{i-2k}$}] (P321) at (5.4,3) [circle,draw,fill=black!100, inner sep=1pt, minimum width=1pt] {};
\node[label={[font=\tiny]right:$b_{i-3k}$}] (P322) at (6.6,3) [circle,draw,fill=black!100, inner sep=1pt, minimum width=1pt] {};
\node[label={[font=\tiny,rotate=90]right:$b_{i+4k}$}] (P1111) at (-6.9,4) [circle,draw,fill=black!100, inner sep=1pt, minimum width=1pt] {};
\node[label={[font=\tiny,rotate=90]right:$a_{i+3k}$}] (P1112) at (-6.3,4) [circle,draw,fill=black!100, inner sep=1pt, minimum width=1pt] {};
\node[label={[font=\tiny,rotate=90]right:$a_{i+2k+1}$}] (P1121) at (-5.7,4) [circle,draw,fill=black!100, inner sep=1pt, minimum width=1pt] {};
\node[label={[font=\tiny,rotate=90]right:$a_{i+2k-1}$}] (P1122) at (-5.1,4) [circle,draw,fill=black!100, inner sep=1pt, minimum width=1pt] {};
\node[label={[font=\tiny,rotate=90]right:$a_{i+k+2}$}] (P1211) at (-4.5,4) [circle,draw,fill=black!100, inner sep=1pt, minimum width=1pt] {};
\node[label={[font=\tiny,rotate=90]right:$b_{i+k+1}$}] (P1212) at (-3.9,4) [diamond,draw,fill=black!100, inner sep=1.75pt, minimum width=1.75pt] {};
\node[label={[font=\tiny,rotate=90]right:$b_{i+k-1}$}] (P1221) at (-3.3,4)  [star,star points=4,star point ratio=2,draw,fill=black!100, inner sep=1.25pt, minimum width=1.25pt] {};
\node[label={[font=\tiny,rotate=90]right:$a_{i+k-2}$}] (P1222) at (-2.7,4) [circle,draw,fill=black!100, inner sep=1pt, minimum width=1pt] {};
\node[label={[font=\tiny,rotate=90]right:$a_{i+3}$}] (P2111) at (-2.1,4) [circle,draw,fill=black!100, inner sep=1pt, minimum width=1pt] {};
\node[label={[font=\tiny,rotate=90]right:$b_{i+2}$}] (P2112) at (-1.5,4) [circle,draw,fill=black!100, inner sep=1pt, minimum width=1pt] {};
\node[label={[font=\tiny,rotate=90]right:$b_{i+k+1}$}] (P2121) at (-.9,4) [diamond,draw,fill=black!100, inner sep=1.75pt, minimum width=1.75pt] {};
\node[label={[font=\tiny,rotate=90]right:$b_{i-k+1}$}] (P2122) at (-.3,4) [star,star points=5,star point ratio=2,draw,fill=black!100, inner sep=1.25pt, minimum width=1.25pt] {};
\node[label={[font=\tiny,rotate=90]right:$b_{i+k-1}$}] (P2211) at (.3,4)  [star,star points=4,star point ratio=2,draw,fill=black!100, inner sep=1.25pt, minimum width=1.25pt] {};
\node[label={[font=\tiny,rotate=90]right:$b_{i-k-1}$}] (P2212) at (.9,4) [regular polygon,regular polygon sides=4,draw,fill=black!100, inner sep=1.75pt, minimum width=1.75pt] {};
\node[label={[font=\tiny,rotate=90]right:$b_{i-2}$}] (P2221) at (1.5,4) [circle,draw,fill=black!100, inner sep=1pt, minimum width=1pt] {};
\node[label={[font=\tiny,rotate=90]right:$a_{i-3}$}] (P2222) at (2.1,4) [circle,draw,fill=black!100, inner sep=1pt, minimum width=1pt] {};
\node[label={[font=\tiny,rotate=90]right:$a_{i-k+2}$}] (P3111) at (2.7,4) [circle,draw,fill=black!100, inner sep=1pt, minimum width=1pt] {};
\node[label={[font=\tiny,rotate=90]right:$b_{i-k+1}$}] (P3112) at (3.3,4) [star,star points=5,star point ratio=2,draw,fill=black!100, inner sep=1.25pt, minimum width=1.25pt] {};
\node[label={[font=\tiny,rotate=90]right:$b_{i-k-1}$}] (P3121) at (3.9,4) [regular polygon,regular polygon sides=4,draw,fill=black!100, inner sep=1.75pt, minimum width=1.75pt] {};
\node[label={[font=\tiny,rotate=90]right:$a_{i-k-2}$}] (P3122) at (4.5,4) [circle,draw,fill=black!100, inner sep=1pt, minimum width=1pt] {};
\node[label={[font=\tiny,rotate=90]right:$a_{i-2k+1}$}] (P3211) at (5.1,4) [circle,draw,fill=black!100, inner sep=1pt, minimum width=1pt] {};
\node[label={[font=\tiny,rotate=90]right:$a_{i-2k-1}$}] (P3212) at (5.7,4) [circle,draw,fill=black!100, inner sep=1pt, minimum width=1pt] {};
\node[label={[font=\tiny,rotate=90]right:$a_{i-3k}$}] (P3221) at (6.3,4) [circle,draw,fill=black!100, inner sep=1pt, minimum width=1pt] {};
\node[label={[font=\tiny,rotate=90]right:$b_{i-4k}$}] (P3222) at (6.9,4) [circle,draw,fill=black!100, inner sep=1pt, minimum width=1pt] {};
\draw (P0)--(P1)--(P11)--(P111)--(P1111);
\draw (P111)--(P1112);
\draw (P11)--(P112)--(P1121);
\draw (P112)--(P1122);
\draw (P1)--(P12)--(P121)--(P1211);
\draw (P121)--(P1212);
\draw (P12)--(P122)--(P1221);
\draw (P122)--(P1222);
\draw (P0)--(P2)--(P21)--(P211)--(P2111);
\draw (P211)--(P2112);
\draw (P21)--(P212)--(P2121);
\draw (P212)--(P2122);
\draw (P2)--(P22)--(P221)--(P2211);
\draw (P221)--(P2212);
\draw (P22)--(P222)--(P2221);
\draw (P222)--(P2222);
\draw (P0)--(P3)--(P31)--(P311)--(P3111);
\draw (P311)--(P3112);
\draw (P31)--(P312)--(P3121);
\draw (P312)--(P3122);
\draw (P3)--(P32)--(P321)--(P3211);
\draw (P321)--(P3212);
\draw (P32)--(P322)--(P3221);
\draw (P322)--(P3222);
\end{tikzpicture}
\end{figure}
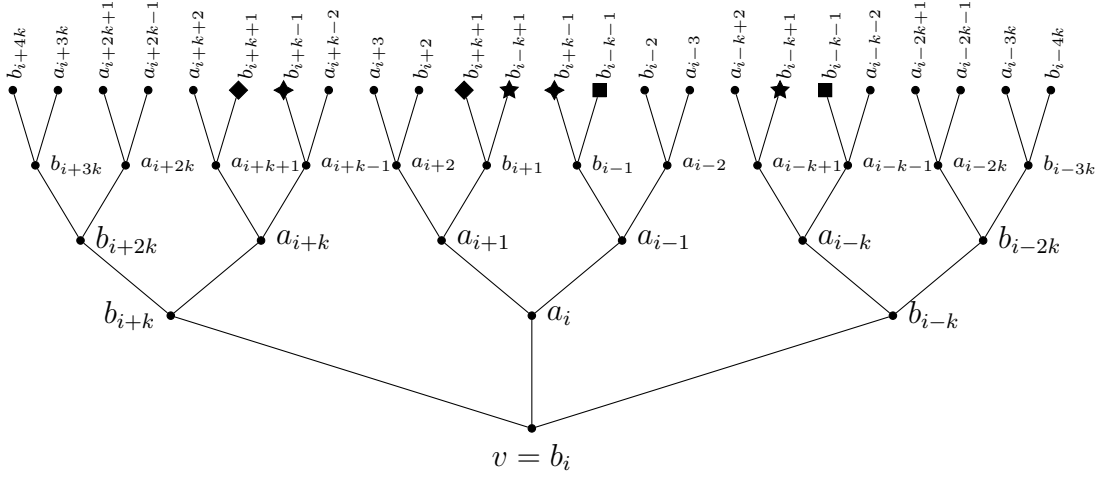

Notice that some vertices at distance $4$ from $v$ are already the same as others. The nodes of identified vertices are illustrated by matching shapes in Figures~\ref{fig:labelA} and~\ref{fig:labelB}. In any generalized Petersen graph, these yield $4$ cycles of length $8$ that include the given vertex $v$. In fact, if $v \in \{a_i,b_i\}$ then each of these cycles includes both $a_i$ and $b_i$. However, no pair of these cycles has for their intersection a path of length $2$ with $v$ at its center (their pairwise intersections all have either $1$ or $3$ edges).  So in order for a $2$-trapped configuration to be possible, some of the other vertices at distance $4$ from $v$ must be identified. 

We make a list of all possible values of $n$ and $k$ that \emph{may} give us the structure we are looking for. These values were found by solving equations modulo $n$ for each possible pair of vertices at distance 4 from $v$ (without pairing vertices in $A$ with vertices in $B$, as they could not possibly be the same). For example, we may have $b_{i-2k+1}=b_{i+2k-1}$, which occurs if $4k-2 \equiv 0 \pmod{n}$. We ignore any solutions that do not meet the definition of a generalized Petersen graph (which requires that $n \ge 5$ and $k < n/2$).

The relationships between $n$ and $k$ that result in additional cycles of length $8$ or less are as follows, when $v \in A$:
\begin{itemize}
\item $k=1,2,3,5$;
\item $k=n-3$ (which can only arise if $n=5$ and $k=2$);
\item $k=n-5$ (which can only arise if $n-5<n/2$, so $n\le 9$);
\item $n=6,8$;
\item $n=2k+i$ where $i \in \{1,2,4\}$;
\item $n=3k+i$ where $i \in \{0, \pm 1, \pm 3\}$;
\item $n=4k+i$ where $i \in \{0, \pm 2\}$;
\item $n=5k\pm 1$ or $n=(5k \pm 1)/2$; and
\item $n=6k$.
\end{itemize}

Comparing this to the values in Table~\ref{tab:girth} and using the smallest value of $k$ from any isomorphism class (from the classes as given in Proposition~\ref{prop:iso}), we see that of these relationships, only the following can arise in generalized Petersen graphs of girth $8$:
\begin{itemize}
\item $k=5$;
\item $n=2k+4$;
\item $n=3k\pm3$; and
\item $n=4k\pm2$.
\end{itemize}

When $v \in B$, the exact same cases arise with the addition of one extra case where $n=8k$. 

Although $n=ak+b$ and $n=ak-b$ produce different graphs, the structures of their cycles of length $8$ as shown in the distance $4$ subgraph from any vertex are identical. Thus, we can illustrate the cycle structure of the cases $n=3k\pm3$ with a single figure, and likewise the cases $n=4k\pm 2$. Additionally, the cycle structure of the $n=3k\pm3$ case is the same regardless of whether $v\in{A}$ or $v\in{B}$, and the cycle structure of $n=2k+4$ where $v\in{A}$ is the same as the cycle structure of $n=4k\pm2$ where $v\in{B}$ and vice versa, so these will also share figures in the following analysis.

The first case we consider is $k=5$ with $v \in A$. Figure~\ref{fig:k=5} shows the structure of the subgraph of vertices at distance at most $4$ from $v$. Cycles of length $8$ whose intersection is a path of length $2$ have been thickened for clarity. In this graph, there are six cycles of length $8$ that can be paired in three different ways to achieve a structure that allows the cops to $2$-trap the robber.
\begin{figure}\caption{$v\in{A}$ and $k=5$}\label{fig:k=5}
\begin{tikzpicture}
\node[label=below:$v$] (P0) at (0,-.5) [circle,draw,fill=black!100, inner sep=1pt, minimum width=1pt] {};
\node (P1) at (-4.8,1) [circle,draw,fill=black!100, inner sep=1pt, minimum width=1pt] {};
\node (P2) at (0,1) [circle,draw,fill=black!100, inner sep=1pt, minimum width=1pt] {};
\node (P3) at (4.8,1) [circle,draw,fill=black!100, inner sep=1pt, minimum width=1pt] {};
\node (P11) at (-6,2) [circle,draw,fill=black!100, inner sep=1pt, minimum width=1pt] {};
\node (P12) at (-3.6,2) [circle,draw,fill=black!100, inner sep=1pt, minimum width=1pt] {};
\node (P21) at (-1.2,2) [circle,draw,fill=black!100, inner sep=1pt, minimum width=1pt] {};
\node (P22) at (1.2,2) [circle,draw,fill=black!100, inner sep=1pt, minimum width=1pt] {};
\node (P31) at (3.6,2) [circle,draw,fill=black!100, inner sep=1pt, minimum width=1pt] {};
\node (P32) at (6,2) [circle,draw,fill=black!100, inner sep=1pt, minimum width=1pt] {};
\node (P111) at (-6.6,3) [circle,draw,fill=black!100, inner sep=1pt, minimum width=1pt] {};
\node (P112) at (-5.4,3) [circle,draw,fill=black!100, inner sep=1pt, minimum width=1pt] {};
\node (P121) at (-4.2,3) [circle,draw,fill=black!100, inner sep=1pt, minimum width=1pt] {};
\node (P122) at (-3,3) [circle,draw,fill=black!100, inner sep=1pt, minimum width=1pt] {};
\node (P211) at (-1.8,3) [circle,draw,fill=black!100, inner sep=1pt, minimum width=1pt] {};
\node (P212) at (-.6,3) [circle,draw,fill=black!100, inner sep=1pt, minimum width=1pt] {};
\node (P221) at (0.6,3) [circle,draw,fill=black!100, inner sep=1pt, minimum width=1pt] {};
\node (P222) at (1.8,3) [circle,draw,fill=black!100, inner sep=1pt, minimum width=1pt] {};
\node (P311) at (3,3) [circle,draw,fill=black!100, inner sep=1pt, minimum width=1pt] {};
\node (P312) at (4.2,3) [circle,draw,fill=black!100, inner sep=1pt, minimum width=1pt] {};
\node (P321) at (5.4,3) [circle,draw,fill=black!100, inner sep=1pt, minimum width=1pt] {};
\node (P322) at (6.6,3) [circle,draw,fill=black!100, inner sep=1pt, minimum width=1pt] {};
\node (P1111) at (-6.9,4) [regular polygon,regular polygon sides=4,draw,fill=black!100, inner sep=2pt, minimum width=2pt] {};
\node (P1112) at (-6.3,4) [regular polygon,regular polygon sides=3,draw,fill=black!100, inner sep=1pt, minimum width=1pt] {};
\node (P1121) at (-5.7,4) [regular polygon,regular polygon sides=3,rotate=180,draw,fill=black!100, inner sep=1pt, minimum width=1pt] {};
\node (P1122) at (-5.1,4) [circle,draw,fill=black!100, inner sep=1pt, minimum width=1pt] {};
\node (P1211) at (-4.5,4) [circle,draw,fill=black!100, inner sep=1pt, minimum width=1pt] {};
\node (P1212) at (-3.9,4)  [diamond,draw,fill=black!100, inner sep=2pt, minimum width=2pt] {};
\node (P1221) at (-3.3,4) [star,star points=4,star point ratio=2,draw,fill=black!100, inner sep=1pt, minimum width=1pt] {};
\node (P1222) at (-2.7,4) [circle,draw,fill=black!100, inner sep=1pt, minimum width=1pt] {};
\node (P2111) at (-2.1,4) [circle,draw,fill=black!100, inner sep=1pt, minimum width=1pt] {};
\node (P2112) at (-1.5,4) [circle,draw,fill=black!100, inner sep=1pt, minimum width=1pt] {};
\node (P2121) at (-.9,4)  [diamond,draw,fill=black!100, inner sep=2pt, minimum width=2pt] {};
\node (P2122) at (-.3,4) [star,star points=5,star point ratio=2,draw,fill=black!100, inner sep=1pt, minimum width=1pt] {};
\node (P2211) at (.3,4)  [star,star points=4,star point ratio=2,draw,fill=black!100, inner sep=1pt, minimum width=1pt] {};
\node (P2212) at (.9,4) [regular polygon,regular polygon sides=4,draw,fill=black!100, inner sep=2pt, minimum width=2pt] {};
\node (P2221) at (1.5,4) [circle,draw,fill=black!100, inner sep=1pt, minimum width=1pt] {};
\node (P2222) at (2.1,4) [circle,draw,fill=black!100, inner sep=1pt, minimum width=1pt] {};
\node (P3111) at (2.7,4) [circle,draw,fill=black!100, inner sep=1pt, minimum width=1pt] {};
\node (P3112) at (3.3,4) [star,star points=5,star point ratio=2,draw,fill=black!100, inner sep=1pt, minimum width=1pt] {};
\node (P3121) at (3.9,4)  [regular polygon,regular polygon sides=4,draw,fill=black!100, inner sep=2pt, minimum width=2pt] {};
\node (P3122) at (4.5,4) [circle,draw,fill=black!100, inner sep=1pt, minimum width=1pt] {};
\node (P3211) at (5.1,4) [circle,draw,fill=black!100, inner sep=1pt, minimum width=1pt] {};
\node (P3212) at (5.7,4) [regular polygon,regular polygon sides=3,draw,fill=black!100, inner sep=1pt, minimum width=1pt] {};
\node (P3221) at (6.3,4) [regular polygon,regular polygon sides=3,rotate=180,draw,fill=black!100, inner sep=1pt, minimum width=1pt] {};
\node (P3222) at (6.9,4)  [diamond,draw,fill=black!100, inner sep=1.75pt, minimum width=1.75pt] {};
\draw (P0)--(P1)--(P11)--(P111)--(P1111);
\draw (P111)--(P1112);
\draw (P11)--(P112)--(P1121);
\draw (P112)--(P1122);
\draw (P1)--(P12)--(P121)--(P1211);
\draw (P121)--(P1212);
\draw (P12)--(P122)--(P1221);
\draw (P122)--(P1222);
\draw (P0)--(P2)--(P21)--(P211)--(P2111);
\draw (P211)--(P2112);
\draw (P21)--(P212)--(P2121);
\draw (P212)--(P2122);
\draw (P2)--(P22)--(P221)--(P2211);
\draw (P221)--(P2212);
\draw (P22)--(P222)--(P2221);
\draw (P222)--(P2222);
\draw (P0)--(P3)--(P31)--(P311)--(P3111);
\draw (P311)--(P3112);
\draw (P31)--(P312)--(P3121);
\draw (P312)--(P3122);
\draw (P3)--(P32)--(P321)--(P3211);
\draw (P321)--(P3212);
\draw (P32)--(P322)--(P3221);
\draw (P322)--(P3222);
\draw[line width=1.2pt] (P0)--(P1)--(P11)--(P111)--(P1111);
\draw[line width=1.2pt] (P0)--(P2)--(P22)--(P221)--(P2212);
\draw[line width=1.2pt] (P0)--(P3)--(P31)--(P312)--(P3121);
\draw[line width=1.2pt] (P0)--(P1)--(P12)--(P121)--(P1212);
\draw[line width=1.2pt] (P0)--(P2)--(P21)--(P212)--(P2121);
\draw[line width=1.2pt] (P0)--(P3)--(P32)--(P322)--(P3222);
\path[dashed] (P1221.north) edge [out=30,in=150] node [right] {} (P2211.north);
\path[dashed, line width=1.2pt] (P1212.north) edge [out=30,in=150] node [right] {} (P2121.north);
\path[dashed, line width=1.2pt] (P2212.north) edge [out=30,in=150] node [right] {} (P3121.north);
\path[dashed] (P2122.north) edge [out=30,in=150] node [right] {} (P3112.north);
\path[dashed, line width=1.2pt] (P1111.north) edge [out=30,in=150] node [right] {} (P2212.north);
\path[dashed, line width=1.2pt] (P2121.north) edge [out=30,in=150] node [right] {} (P3222.north);
\path[dashed] (P1112.north) edge [out=30,in=150] node [right] {} (P3212.north);
\path[dashed] (P1121.south) edge [out=30,in=150] node [right] {} (P3221.south);
\end{tikzpicture}
\end{figure}
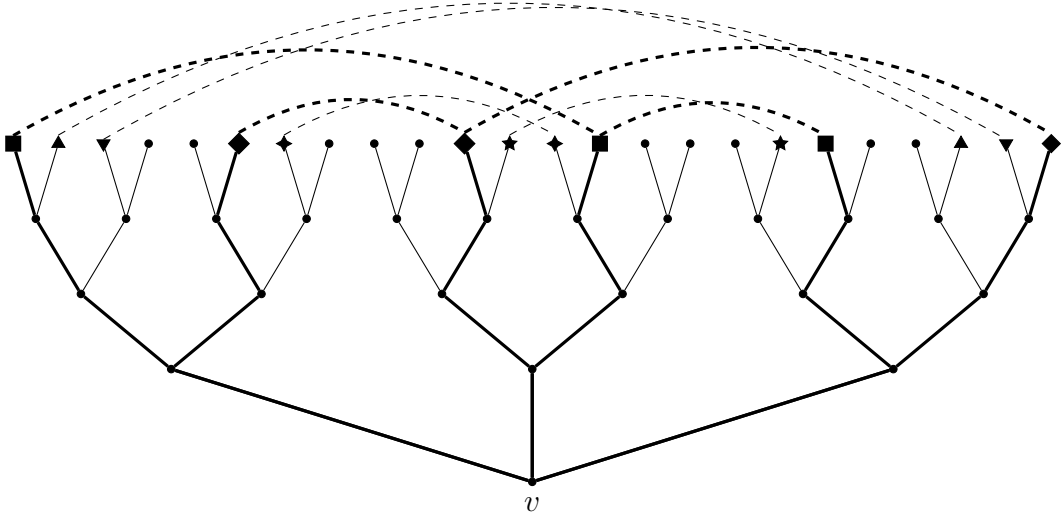

When $k=5$ and $v\in{B}$ we do not find two cycles of length $8$ whose intersection is a path of length $2$. However, 
 this simply means that the cops cannot $2$-trap the robber when the robber is on a $B$ vertex in a graph with $k=5$. Since we have just seen that the cops can $2$-trap the robber on an $A$ vertex, this family of generalized Petersen graphs must be included as possible exceptions.

The next case to consider is $n=2k+4$ with $v\in{A}$, which is shown in Figure~\ref{fig:2k+4A}. As previously noted, this figure also covers the case $n=4k\pm 2$ with $v \in B$. In this case, there are three possible ways to make cycles of length $8$ that intersect in paths of length $2$. One is to use the vertices indicated by the pentagon and the upside down pentagon, another is to use the square and the upside down triangle. The third way (which was not made bold, to avoid too many bold lines) is the reflection of the previous way, where we use the vertices indicated by the triangle and the diamond.

\begin{figure}\caption{Sub tree structure when $v\in{A}$ and $n=2k+4$, and when $v\in{B}$ and $n=4k\pm2$}\label{fig:2k+4A}
\begin{tikzpicture}
\node[label=below:$v$] (P0) at (0,-.5) [circle,draw,fill=black!100, inner sep=1pt, minimum width=1pt] {};
\node (P1) at (-4.8,1) [circle,draw,fill=black!100, inner sep=1pt, minimum width=1pt] {};
\node (P2) at (0,1) [circle,draw,fill=black!100, inner sep=1pt, minimum width=1pt] {};
\node (P3) at (4.8,1) [circle,draw,fill=black!100, inner sep=1pt, minimum width=1pt] {};
\node (P11) at (-6,2) [circle,draw,fill=black!100, inner sep=1pt, minimum width=1pt] {};
\node (P12) at (-3.6,2) [circle,draw,fill=black!100, inner sep=1pt, minimum width=1pt] {};
\node (P21) at (-1.2,2) [circle,draw,fill=black!100, inner sep=1pt, minimum width=1pt] {};
\node (P22) at (1.2,2) [circle,draw,fill=black!100, inner sep=1pt, minimum width=1pt] {};
\node (P31) at (3.6,2) [circle,draw,fill=black!100, inner sep=1pt, minimum width=1pt] {};
\node (P32) at (6,2) [circle,draw,fill=black!100, inner sep=1pt, minimum width=1pt] {};
\node (P111) at (-6.6,3) [circle,draw,fill=black!100, inner sep=1pt, minimum width=1pt] {};
\node (P112) at (-5.4,3) [circle,draw,fill=black!100, inner sep=1pt, minimum width=1pt] {};
\node (P121) at (-4.2,3) [circle,draw,fill=black!100, inner sep=1pt, minimum width=1pt] {};
\node (P122) at (-3,3) [circle,draw,fill=black!100, inner sep=1pt, minimum width=1pt] {};
\node (P211) at (-1.8,3) [circle,draw,fill=black!100, inner sep=1pt, minimum width=1pt] {};
\node (P212) at (-.6,3) [circle,draw,fill=black!100, inner sep=1pt, minimum width=1pt] {};
\node (P221) at (0.6,3) [circle,draw,fill=black!100, inner sep=1pt, minimum width=1pt] {};
\node (P222) at (1.8,3) [circle,draw,fill=black!100, inner sep=1pt, minimum width=1pt] {};
\node (P311) at (3,3) [circle,draw,fill=black!100, inner sep=1pt, minimum width=1pt] {};
\node (P312) at (4.2,3) [circle,draw,fill=black!100, inner sep=1pt, minimum width=1pt] {};
\node (P321) at (5.4,3) [circle,draw,fill=black!100, inner sep=1pt, minimum width=1pt] {};
\node (P322) at (6.6,3) [circle,draw,fill=black!100, inner sep=1pt, minimum width=1pt] {};
\node (P1111) at (-6.9,4) [regular polygon,regular polygon sides=3,draw,fill=black!100, inner sep=1pt, minimum width=1pt] {};
\node (P1112) at (-6.3,4)  [regular polygon,regular polygon sides=5,draw,fill=black!100, inner sep=2pt, minimum width=2pt] {};
\node (P1121) at (-5.7,4) [star,star points=6,star point ratio=2,draw,fill=black!100, inner sep=1pt, minimum width=1pt] {};
\node (P1122) at (-5.1,4) [circle,draw,fill=black!100, inner sep=1pt, minimum width=1pt] {};
\node (P1211) at (-4.5,4)  [regular polygon,regular polygon sides=5,rotate=180,draw,fill=black!100, inner sep=2pt, minimum width=2pt] {};
\node (P1212) at (-3.9,4) [diamond,draw,fill=black!100, inner sep=1.25pt, minimum width=1.25pt] {};
\node (P1221) at (-3.3,4) [star,star points=4,star point ratio=2,draw,fill=black!100, inner sep=1pt, minimum width=1pt] {};
\node (P1222) at (-2.7,4) [circle,draw,fill=black!100, inner sep=1pt, minimum width=1pt] {};
\node (P2111) at (-2.1,4) [circle,draw,fill=black!100, inner sep=1pt, minimum width=1pt] {};
\node (P2112) at (-1.5,4) [regular polygon,regular polygon sides=3,rotate=180,draw,fill=black!100, inner sep=1.5pt, minimum width=1.5pt] {};
\node (P2121) at (-.9,4) [diamond,draw,fill=black!100, inner sep=1.25pt, minimum width=1.25pt] {};
\node (P2122) at (-.3,4) [star,star points=5,star point ratio=2,draw,fill=black!100, inner sep=1pt, minimum width=1pt] {};
\node (P2211) at (.3,4) [star,star points=4,star point ratio=2,draw,fill=black!100, inner sep=1pt, minimum width=1pt] {};
\node (P2212) at (.9,4) [regular polygon,regular polygon sides=4,draw,fill=black!100, inner sep=2pt, minimum width=2pt] {};
\node (P2221) at (1.5,4) [regular polygon,regular polygon sides=3,draw,fill=black!100, inner sep=1pt, minimum width=1pt] {};
\node (P2222) at (2.1,4) [circle,draw,fill=black!100, inner sep=1pt, minimum width=1pt] {};
\node (P3111) at (2.7,4) [circle,draw,fill=black!100, inner sep=1pt, minimum width=1pt] {};
\node (P3112) at (3.3,4) [star,star points=5,star point ratio=2,draw,fill=black!100, inner sep=1pt, minimum width=1pt] {};
\node (P3121) at (3.9,4) [regular polygon,regular polygon sides=4,draw,fill=black!100, inner sep=2pt, minimum width=2pt] {};
\node (P3122) at (4.5,4) [regular polygon,regular polygon sides=5,draw,fill=black!100, inner sep=2pt, minimum width=2pt] {};
\node (P3211) at (5.1,4) [circle,draw,fill=black!100, inner sep=1pt, minimum width=1pt] {};
\node (P3212) at (5.7,4)  [star,star points=6,star point ratio=2,draw,fill=black!100, inner sep=1pt, minimum width=1pt] {};
\node (P3221) at (6.3,4) [regular polygon,regular polygon sides=5,rotate=180,draw,fill=black!100, inner sep=2pt, minimum width=2pt] {};
\node (P3222) at (6.9,4) [regular polygon,regular polygon sides=3,rotate=180,draw,fill=black!100, inner sep=1.5pt, minimum width=1.5pt] {};
\draw (P0)--(P1)--(P11)--(P111)--(P1111);
\draw (P111)--(P1112);
\draw (P11)--(P112)--(P1121);
\draw (P112)--(P1122);
\draw (P1)--(P12)--(P121)--(P1211);
\draw (P121)--(P1212);
\draw (P12)--(P122)--(P1221);
\draw (P122)--(P1222);
\draw (P0)--(P2)--(P21)--(P211)--(P2111);
\draw (P211)--(P2112);
\draw (P21)--(P212)--(P2121);
\draw (P212)--(P2122);
\draw (P2)--(P22)--(P221)--(P2211);
\draw (P221)--(P2212);
\draw (P22)--(P222)--(P2221);
\draw (P222)--(P2222);
\draw (P0)--(P3)--(P31)--(P311)--(P3111);
\draw (P311)--(P3112);
\draw (P31)--(P312)--(P3121);
\draw (P312)--(P3122);
\draw (P3)--(P32)--(P321)--(P3211);
\draw (P321)--(P3212);
\draw (P32)--(P322)--(P3221);
\draw (P322)--(P3222);
\draw[line width=1.2pt] (P0)--(P1)--(P11)--(P111)--(P1112);
\draw[line width=1.2pt] (P0)--(P1)--(P12)--(P121)--(P1211);
\draw[line width=1.2pt] (P0)--(P2)--(P21)--(P211)--(P2112);
\draw[line width=1.2pt] (P0)--(P2)--(P22)--(P221)--(P2212);
\draw[line width=1.2pt] (P0)--(P3)--(P31)--(P312)--(P3121);
\draw[line width=1.2pt] (P0)--(P3)--(P31)--(P312)--(P3122);
\draw[line width=1.2pt] (P0)--(P3)--(P32)--(P322)--(P3221);
\draw[line width=1.2pt] (P0)--(P3)--(P32)--(P322)--(P3222);
\path[dashed] (P1221.north) edge [out=30,in=150] node [right] {} (P2211.north);
\path[dashed] (P1212.north) edge [out=30,in=150] node [right] {} (P2121.north);
\path[dashed, line width=1.2pt] (P2212.north) edge [out=30,in=150] node [right] {} (P3121.north);
\path[dashed] (P2122.north) edge [out=30,in=150] node [right] {} (P3112.north);
\path[dashed, line width=1.2pt] (P1112.north) edge [out=30,in=150] node [right] {} (P3122.north);
\path[dashed, line width=1.2pt] (P1211.south) edge [out=30,in=150] node [right] {} (P3221.south);
\path[dashed] (P1111.north) edge [out=30,in=150] node [right] {} (P2221.north);
\path[dashed, line width=1.2pt] (P2112.south) edge [out=30,in=150] node [right] {} (P3222.south);
\path[dashed] (P1121.north) edge [out=30,in=150] node [right] {} (P3212.north);
\end{tikzpicture}
\end{figure}
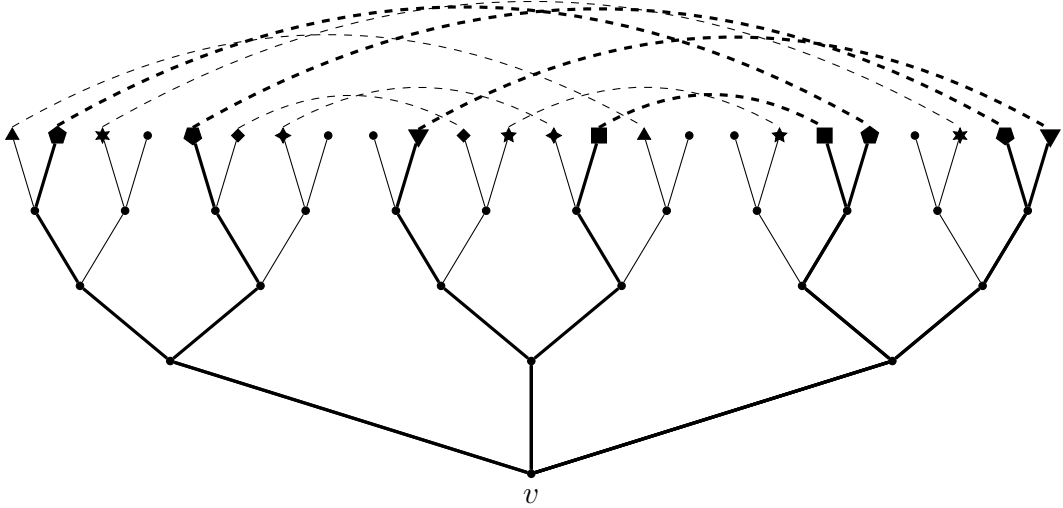

When $n=2k+4$ and $v\in{B}$ or $n=4k\pm 2$ and $v \in A$, the cycle structure is shown in Figure~\ref{fig:2k+4B}. There are two ways to make cycles of length $8$ that intersect in paths of length $2$. We can use the vertices indicated by the triangle and the diamond, or we can use the vertices indicated by the upside down triangle and the square. 

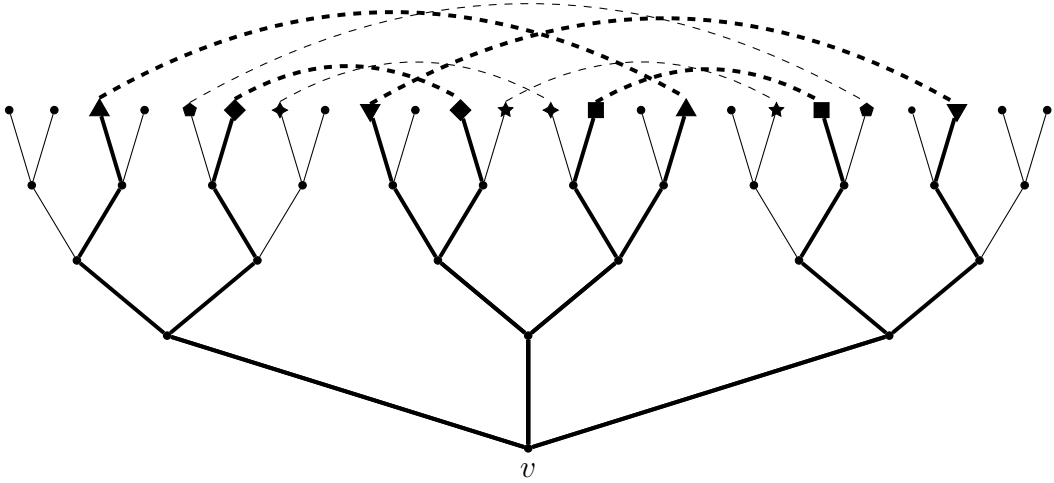
\begin{figure}\caption{Sub tree structure when $v\in{B}$ and $n=2k+4$, and when $v\in{A}$ and $n=4k\pm2$}\label{fig:2k+4B}
\begin{tikzpicture}
\node[label=below:$v$] (P0) at (0,-.5) [circle,draw,fill=black!100, inner sep=1pt, minimum width=1pt] {};
\node (P1) at (-4.8,1) [circle,draw,fill=black!100, inner sep=1pt, minimum width=1pt] {};
\node (P2) at (0,1) [circle,draw,fill=black!100, inner sep=1pt, minimum width=1pt] {};
\node (P3) at (4.8,1) [circle,draw,fill=black!100, inner sep=1pt, minimum width=1pt] {};
\node (P11) at (-6,2) [circle,draw,fill=black!100, inner sep=1pt, minimum width=1pt] {};
\node (P12) at (-3.6,2) [circle,draw,fill=black!100, inner sep=1pt, minimum width=1pt] {};
\node (P21) at (-1.2,2) [circle,draw,fill=black!100, inner sep=1pt, minimum width=1pt] {};
\node (P22) at (1.2,2) [circle,draw,fill=black!100, inner sep=1pt, minimum width=1pt] {};
\node (P31) at (3.6,2) [circle,draw,fill=black!100, inner sep=1pt, minimum width=1pt] {};
\node (P32) at (6,2) [circle,draw,fill=black!100, inner sep=1pt, minimum width=1pt] {};
\node (P111) at (-6.6,3) [circle,draw,fill=black!100, inner sep=1pt, minimum width=1pt] {};
\node (P112) at (-5.4,3) [circle,draw,fill=black!100, inner sep=1pt, minimum width=1pt] {};
\node (P121) at (-4.2,3) [circle,draw,fill=black!100, inner sep=1pt, minimum width=1pt] {};
\node (P122) at (-3,3) [circle,draw,fill=black!100, inner sep=1pt, minimum width=1pt] {};
\node (P211) at (-1.8,3) [circle,draw,fill=black!100, inner sep=1pt, minimum width=1pt] {};
\node (P212) at (-.6,3) [circle,draw,fill=black!100, inner sep=1pt, minimum width=1pt] {};
\node (P221) at (0.6,3) [circle,draw,fill=black!100, inner sep=1pt, minimum width=1pt] {};
\node (P222) at (1.8,3) [circle,draw,fill=black!100, inner sep=1pt, minimum width=1pt] {};
\node (P311) at (3,3) [circle,draw,fill=black!100, inner sep=1pt, minimum width=1pt] {};
\node (P312) at (4.2,3) [circle,draw,fill=black!100, inner sep=1pt, minimum width=1pt] {};
\node (P321) at (5.4,3) [circle,draw,fill=black!100, inner sep=1pt, minimum width=1pt] {};
\node (P322) at (6.6,3) [circle,draw,fill=black!100, inner sep=1pt, minimum width=1pt] {};
\node (P1111) at (-6.9,4) [circle,draw,fill=black!100, inner sep=1pt, minimum width=1pt] {};
\node (P1112) at (-6.3,4) [circle,draw,fill=black!100, inner sep=1pt, minimum width=1pt] {};
\node (P1121) at (-5.7,4) [regular polygon,regular polygon sides=3,draw,fill=black!100, inner sep=1.5pt, minimum width=1.5pt] {};
\node (P1122) at (-5.1,4) [circle,draw,fill=black!100, inner sep=1pt, minimum width=1pt] {};
\node (P1211) at (-4.5,4) [regular polygon,regular polygon sides=5,draw,fill=black!100, inner sep=1.5pt, minimum width=1.5pt] {};
\node (P1212) at (-3.9,4) [diamond,draw,fill=black!100, inner sep=2pt, minimum width=2pt] {};
\node (P1221) at (-3.3,4) [star,star points=4,star point ratio=2,draw,fill=black!100, inner sep=1pt, minimum width=1pt] {};
\node (P1222) at (-2.7,4) [circle,draw,fill=black!100, inner sep=1pt, minimum width=1pt] {};
\node (P2111) at (-2.1,4) [regular polygon,regular polygon sides=3,rotate=180,draw,fill=black!100, inner sep=1.5pt, minimum width=1.5pt] {};
\node (P2112) at (-1.5,4) [circle,draw,fill=black!100, inner sep=1pt, minimum width=1pt] {};
\node (P2121) at (-.9,4) [diamond,draw,fill=black!100, inner sep=2pt, minimum width=2pt] {};
\node (P2122) at (-.3,4) [star,star points=5,star point ratio=2,draw,fill=black!100, inner sep=1pt, minimum width=1pt] {};
\node (P2211) at (.3,4) [star,star points=4,star point ratio=2,draw,fill=black!100, inner sep=1pt, minimum width=1pt] {};
\node (P2212) at (.9,4) [regular polygon,regular polygon sides=4,draw,fill=black!100, inner sep=2pt, minimum width=2pt] {};
\node (P2221) at (1.5,4) [circle,draw,fill=black!100, inner sep=1pt, minimum width=1pt] {};
\node (P2222) at (2.1,4) [regular polygon,regular polygon sides=3,draw,fill=black!100, inner sep=1.5pt, minimum width=1.5pt] {};
\node (P3111) at (2.7,4) [circle,draw,fill=black!100, inner sep=1pt, minimum width=1pt] {};
\node (P3112) at (3.3,4) [star,star points=5,star point ratio=2,draw,fill=black!100, inner sep=1pt, minimum width=1pt] {};
\node (P3121) at (3.9,4) [regular polygon,regular polygon sides=4,draw,fill=black!100, inner sep=2pt, minimum width=2pt] {};
\node (P3122) at (4.5,4) [regular polygon,regular polygon sides=5,draw,fill=black!100, inner sep=1.5pt, minimum width=1.5pt] {};
\node (P3211) at (5.1,4) [circle,fill=black!100, inner sep=1pt, minimum width=1pt] {};
\node (P3212) at (5.7,4) [regular polygon,regular polygon sides=3,rotate=180,draw,fill=black!100, inner sep=1.5pt, minimum width=1.5pt] {};
\node (P3221) at (6.3,4) [circle,draw,fill=black!100, inner sep=1pt, minimum width=1pt] {};
\node (P3222) at (6.9,4) [circle,draw,fill=black!100, inner sep=1pt, minimum width=1pt] {};
\draw (P0)--(P1)--(P11)--(P111)--(P1111);
\draw (P111)--(P1112);
\draw (P11)--(P112)--(P1121);
\draw (P112)--(P1122);
\draw (P1)--(P12)--(P121)--(P1211);
\draw (P121)--(P1212);
\draw (P12)--(P122)--(P1221);
\draw (P122)--(P1222);
\draw (P0)--(P2)--(P21)--(P211)--(P2111);
\draw (P211)--(P2112);
\draw (P21)--(P212)--(P2121);
\draw (P212)--(P2122);
\draw (P2)--(P22)--(P221)--(P2211);
\draw (P221)--(P2212);
\draw (P22)--(P222)--(P2221);
\draw (P222)--(P2222);
\draw (P0)--(P3)--(P31)--(P311)--(P3111);
\draw (P311)--(P3112);
\draw (P31)--(P312)--(P3121);
\draw (P312)--(P3122);
\draw (P3)--(P32)--(P321)--(P3211);
\draw (P321)--(P3212);
\draw (P32)--(P322)--(P3221);
\draw (P322)--(P3222);
\draw[line width=1.5pt] (P0)--(P1)--(P11)--(P112)--(P1121);
\draw[line width=1.5pt] (P0)--(P1)--(P12)--(P121)--(P1212);
\draw[line width=1.5pt] (P0)--(P2)--(P21)--(P211)--(P2111);
\draw[line width=1.5pt] (P0)--(P2)--(P21)--(P212)--(P2121);
\draw[line width=1.5pt] (P0)--(P2)--(P22)--(P221)--(P2212);
\draw[line width=1.5pt] (P0)--(P2)--(P22)--(P222)--(P2222);
\draw[line width=1.5pt] (P0)--(P3)--(P31)--(P312)--(P3121);
\draw[line width=1.5pt] (P0)--(P3)--(P32)--(P321)--(P3212);
\path[dashed] (P1221.north) edge [out=30,in=150] node [right] {} (P2211.north);
\path[dashed,line width=1.5pt] (P1212.north) edge [out=30,in=150] node [right] {} (P2121.north);
\path[dashed,line width=1.5pt] (P2212.north) edge [out=30,in=150] node [right] {} (P3121.north);
\path[dashed] (P2122.north) edge [out=30,in=150] node [right] {} (P3112.north);
\path[dashed,line width=1.5pt] (P1121.north) edge [out=30,in=150] node [right] {} (P2222.north);
\path[dashed,line width=1.5pt] (P2111.south) edge [out=30,in=150] node [right] {} (P3212.south);
\path[dashed] (P1211.north) edge [out=30,in=150] node [right] {} (P3122.north);
\end{tikzpicture}
\end{figure}

Our next case is $n=3k\pm3$, shown in Figure~\ref{fig:3k-3}. Whether $v$ is in $A$ or $B$ makes no difference to the structure of the cycles. In this case there are three ways to make cycles of length $8$ that intersect in paths of length $2$. We can use the vertices indicated by the pentagon and the upside down pentagon, or we can use the vertices indicated by the triangle and diamond. We can also use the vertices indicated by the upside down triangle and square, but as in the previous figure, this has not been made bold to avoid making almost the whole figure bold.

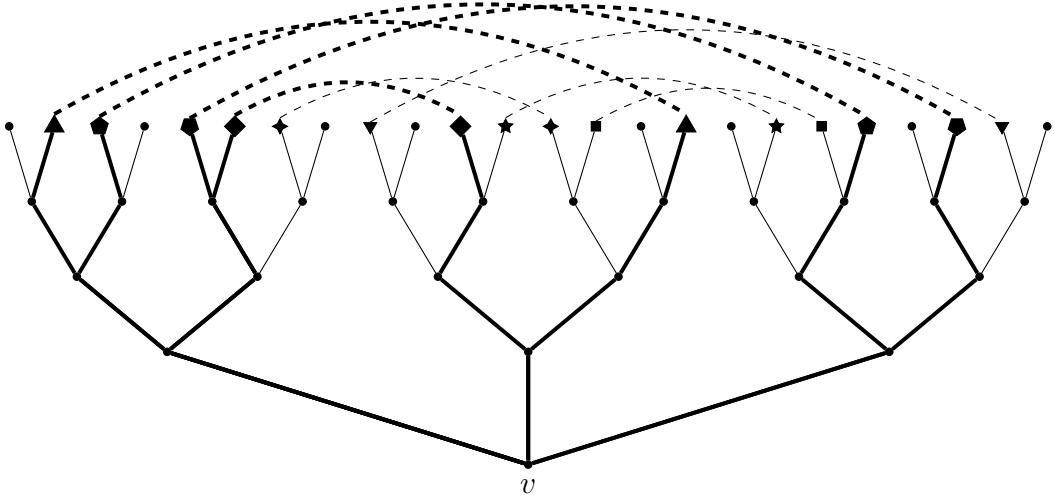
\begin{figure}\caption{$n=3k\pm3$}\label{fig:3k-3}
\begin{tikzpicture}
\node[label=below:$v$] (P0) at (0,-.5) [circle,draw,fill=black!100, inner sep=1pt, minimum width=1pt] {};
\node (P1) at (-4.8,1) [circle,draw,fill=black!100, inner sep=1pt, minimum width=1pt] {};
\node (P2) at (0,1) [circle,draw,fill=black!100, inner sep=1pt, minimum width=1pt] {};
\node (P3) at (4.8,1) [circle,draw,fill=black!100, inner sep=1pt, minimum width=1pt] {};
\node (P11) at (-6,2) [circle,draw,fill=black!100, inner sep=1pt, minimum width=1pt] {};
\node (P12) at (-3.6,2) [circle,draw,fill=black!100, inner sep=1pt, minimum width=1pt] {};
\node (P21) at (-1.2,2) [circle,draw,fill=black!100, inner sep=1pt, minimum width=1pt] {};
\node (P22) at (1.2,2) [circle,draw,fill=black!100, inner sep=1pt, minimum width=1pt] {};
\node (P31) at (3.6,2) [circle,draw,fill=black!100, inner sep=1pt, minimum width=1pt] {};
\node (P32) at (6,2) [circle,draw,fill=black!100, inner sep=1pt, minimum width=1pt] {};
\node (P111) at (-6.6,3) [circle,draw,fill=black!100, inner sep=1pt, minimum width=1pt] {};
\node (P112) at (-5.4,3) [circle,draw,fill=black!100, inner sep=1pt, minimum width=1pt] {};
\node (P121) at (-4.2,3) [circle,draw,fill=black!100, inner sep=1pt, minimum width=1pt] {};
\node (P122) at (-3,3) [circle,draw,fill=black!100, inner sep=1pt, minimum width=1pt] {};
\node (P211) at (-1.8,3) [circle,draw,fill=black!100, inner sep=1pt, minimum width=1pt] {};
\node (P212) at (-.6,3) [circle,draw,fill=black!100, inner sep=1pt, minimum width=1pt] {};
\node (P221) at (0.6,3) [circle,draw,fill=black!100, inner sep=1pt, minimum width=1pt] {};
\node (P222) at (1.8,3) [circle,draw,fill=black!100, inner sep=1pt, minimum width=1pt] {};
\node (P311) at (3,3) [circle,draw,fill=black!100, inner sep=1pt, minimum width=1pt] {};
\node (P312) at (4.2,3) [circle,draw,fill=black!100, inner sep=1pt, minimum width=1pt] {};
\node (P321) at (5.4,3) [circle,draw,fill=black!100, inner sep=1pt, minimum width=1pt] {};
\node (P322) at (6.6,3) [circle,draw,fill=black!100, inner sep=1pt, minimum width=1pt] {};
\node (P1111) at (-6.9,4) [circle,draw,fill=black!100, inner sep=1pt, minimum width=1pt] {};
\node (P1112) at (-6.3,4) [regular polygon,regular polygon sides=3,draw,fill=black!100, inner sep=1.5pt, minimum width=1.5pt] {};
\node (P1121) at (-5.7,4)  [regular polygon,regular polygon sides=5,draw,fill=black!100, inner sep=2pt, minimum width=2pt] {};
\node (P1122) at (-5.1,4) [circle,draw,fill=black!100, inner sep=1pt, minimum width=1pt] {};
\node (P1211) at (-4.5,4)  [regular polygon,regular polygon sides=5,rotate=180,draw,fill=black!100, inner sep=2pt, minimum width=2pt] {};
\node (P1212) at (-3.9,4) [diamond,draw,fill=black!100, inner sep=2pt, minimum width=2pt] {};
\node (P1221) at (-3.3,4) [star,star points=4,star point ratio=2,draw,fill=black!100, inner sep=1pt, minimum width=1pt] {};
\node (P1222) at (-2.7,4) [circle,draw,fill=black!100, inner sep=1pt, minimum width=1pt] {};
\node (P2111) at (-2.1,4) [regular polygon,regular polygon sides=3,rotate=180,draw,fill=black!100, inner sep=1pt, minimum width=1pt] {};
\node (P2112) at (-1.5,4) [circle,draw,fill=black!100, inner sep=1pt, minimum width=1pt] {};
\node (P2121) at (-.9,4) [diamond,draw,fill=black!100, inner sep=2pt, minimum width=2pt] {};
\node (P2122) at (-.3,4) [star,star points=5,star point ratio=2,draw,fill=black!100, inner sep=1pt, minimum width=1pt] {};
\node (P2211) at (.3,4) [star,star points=4,star point ratio=2,draw,fill=black!100, inner sep=1pt, minimum width=1pt] {};
\node (P2212) at (.9,4) [regular polygon,regular polygon sides=4,draw,fill=black!100, inner sep=1.25pt, minimum width=1.25pt] {};
\node (P2221) at (1.5,4) [circle,draw,fill=black!100, inner sep=1pt, minimum width=1pt] {};
\node (P2222) at (2.1,4) [regular polygon,regular polygon sides=3,draw,fill=black!100, inner sep=1.5pt, minimum width=1.5pt] {};
\node (P3111) at (2.7,4) [circle,draw,fill=black!100, inner sep=1pt, minimum width=1pt] {};
\node (P3112) at (3.3,4) [star,star points=5,star point ratio=2,draw,fill=black!100, inner sep=1pt, minimum width=1pt] {};
\node (P3121) at (3.9,4) [regular polygon,regular polygon sides=4,draw,fill=black!100, inner sep=1.25pt, minimum width=1.25pt] {};
\node (P3122) at (4.5,4)  [regular polygon,regular polygon sides=5,draw,fill=black!100, inner sep=2pt, minimum width=2pt] {};
\node (P3211) at (5.1,4) [circle,draw,fill=black!100, inner sep=1pt, minimum width=1pt] {};
\node (P3212) at (5.7,4)  [regular polygon,regular polygon sides=5,rotate=180,draw,fill=black!100, inner sep=2pt, minimum width=2pt] {};
\node (P3221) at (6.3,4) [regular polygon,regular polygon sides=3,rotate=180,draw,fill=black!100, inner sep=1pt, minimum width=1pt] {};
\node (P3222) at (6.9,4) [circle,draw,fill=black!100, inner sep=1pt, minimum width=1pt] {};
\draw (P0)--(P1)--(P11)--(P111)--(P1111);
\draw (P111)--(P1112);
\draw (P11)--(P112)--(P1121);
\draw (P112)--(P1122);
\draw (P1)--(P12)--(P121)--(P1211);
\draw (P121)--(P1212);
\draw (P12)--(P122)--(P1221);
\draw (P122)--(P1222);
\draw (P0)--(P2)--(P21)--(P211)--(P2111);
\draw (P211)--(P2112);
\draw (P21)--(P212)--(P2121);
\draw (P212)--(P2122);
\draw (P2)--(P22)--(P221)--(P2211);
\draw (P221)--(P2212);
\draw (P22)--(P222)--(P2221);
\draw (P222)--(P2222);
\draw (P0)--(P3)--(P31)--(P311)--(P3111);
\draw (P311)--(P3112);
\draw (P31)--(P312)--(P3121);
\draw (P312)--(P3122);
\draw (P3)--(P32)--(P321)--(P3211);
\draw (P321)--(P3212);
\draw (P32)--(P322)--(P3221);
\draw (P322)--(P3222);
\draw[line width=1.5pt] (P0)--(P1)--(P11)--(P111)--(P1112);
\draw[line width=1.5pt] (P0)--(P1)--(P11)--(P112)--(P1121);
\draw[line width=1.5pt] (P0)--(P1)--(P12)--(P121)--(P1211);
\draw[line width=1.5pt] (P0)--(P1)--(P12)--(P121)--(P1212);
\draw[line width=1.5pt] (P0)--(P2)--(P21)--(P212)--(P2121);
\draw[line width=1.5pt] (P0)--(P2)--(P22)--(P222)--(P2222);
\draw[line width=1.5pt] (P0)--(P3)--(P31)--(P312)--(P3122);
\draw[line width=1.5pt] (P0)--(P3)--(P32)--(P321)--(P3212);
\path[dashed] (P1221.north) edge [out=30,in=150] node [right] {} (P2211.north);
\path[dashed,line width=1.5pt] (P1212.north) edge [out=30,in=150] node [right] {} (P2121.north);
\path[dashed] (P2212.north) edge [out=30,in=150] node [right] {} (P3121.north);
\path[dashed] (P2122.north) edge [out=30,in=150] node [right] {} (P3112.north);
\path[dashed,line width=1.5pt] (P1121.north) edge [out=30,in=150] node [right] {} (P3122.north);
\path[dashed,line width=1.5pt] (P1211.south) edge [out=30,in=150] node [right] {} (P3212.south);
\path[dashed,line width=1.5pt] (P1112.north) edge [out=30,in=150] node [right] {} (P2222.north);
\path[dashed] (P2111.south) edge [out=30,in=150] node [right] {} (P3221.south);
\end{tikzpicture}
\end{figure}

The final case to consider is the $n=8k$ case, which only results in extra identified vertices at distance $4$ from $v$ if $v\in{B}$. However, this case does not result in any cycles of length $8$ whose intersection is a path of length $2$.

With the above analysis, we have proved the following lemma. Its corollaries follow immediately from combining this lemma with Theorem~\ref{thm:main}.

\begin{lem}
Suppose that $GP(n,k)$ has girth $8$, and includes two cycles of length $8$ whose intersection is a path of length $2$. Then up to isomorphism, its parameters have one of the following forms:
\begin{itemize}
		\item $k=5$;
		\item $n=2k+4$;
		\item $n=3k\pm3$; or
		\item $n=4k\pm2$.
\end{itemize}
\end{lem}

\begin{cor}\label{cor:girth8}
Suppose that a generalized Petersen graph $GP(n,k)$ has girth 8 and is presented with $k$ as small as possible in its isomorphism class. Then the cop number $c(GP(n,k))=4$ except possibly if one of the following is true:
	\begin{itemize}
		\item $k=5$;
		\item $n=2k+4$;
		\item $n=3k\pm3$; or
		\item $n=4k\pm2$.
	\end{itemize}
\end{cor}

Using the information from Table~\ref{tab:girth}, we can replace our hypothesis about the girth by adding additional parameters to the forbidden list.

\begin{cor}\label{cor:final}
Suppose that a generalized Petersen graph $G(n,k)$ is presented with $k$ as small as possible in its isomorphism class. Then the cop number $c(GP(n,k))=4$ except possibly if one of the following is true:
\begin{itemize}
\item $k \in \{1,2,3,4,5\}$;
\item $n=2k+i$ with $i \in \{2,3,4\}$;
\item $n=3k +i$ with $i \in \{0, \pm 2, \pm 3\}$;
\item $n=4k +i$ with $i \in \{0, \pm 2\}$;
\item $n=5k/i$ with $i \in \{1,2\}$;
\item $n=6k$; or
\item $n=7k/i$ with $i \in \{1,2,3\}$.	
\end{itemize}

\end{cor}

Our results do not guarantee the cop number is 3 or less in any of these cases; indeed, from Table~\ref{tab:table2} we know that there are generalized Petersen graphs with girth $7$ (which are therefore included in the above list) that also have cop number $4$.  All we really know (except where other results apply, such as the cop number when $k=1$) about the cases that are listed as exceptional in Corollary~\ref{cor:final} is that they either have girth less than $8$, or they admit a configuration in which the cops could theoretically $2$-trap the robber.
 
For generalized Petersen graphs with $n \le 40$, the cop numbers are known \cite{ball}. (Although a complete list does not appear in \cite{ball}, the information can be completed using their algorithm \cite{bell}.) We can also find the cop number for all generalized Petersen graphs with $n \le 30$ in \cite{burgess}. While it might be of interest to provide more information about all of these graphs, for the sake of brevity we mention only some highlights here. A table that includes all girths is included as an appendix in the ar$\chi$iv version of this paper. 

None of the graphs of girth $8$ that appear in Table~\ref{tab:table2} have parameters that satisfy any of the conditions of Corollary~\ref{cor:girth8}. So it is possible (though we would not dare to conjecture it) that our result explains all graphs of girth $8$ that have cop number $4$. 

The three graphs of girth $7$ that have cop number $4$ are precisely the graphs whose parameters have the form $n=7k/i$ where $i \in \{2,3\}$. It seems likely that this family of parameters results in different behaviour with respect to the game of cops and robbers, than the other families of parameters that give rise to graphs of girth $7$. 

As mentioned previously, a generalized Petersen graph cannot have cop number $1$, and can only have cop number $2$ if its girth is $3$ or $4$. Thus, if $n \notin \{3k,4k\}$ and $k\neq 1$, then $c(GP(n,k)) \in \{3,4\}$.
All of the generalized Petersen graphs with $n \le 40$ that have cop number $2$ are listed in \cite{ball}, although this is not clear from what they write. These graphs are every graph with $k=1$ (this is a theoretical result proven in \cite{ball}), and the graphs with $n=3k$ or $n=4k$ for $k \in \{2,3\}$. We conjecture that these are the only generalized Petersen graphs that have cop number $2$. 

\section{Conclusion}\label{sec: Further research}

In this paper we focused on generalized Peterson graphs of girth 8. Given that such graphs can have a girth $g$ with $3\le g \le 8$, further research on cops and robbers played on generalized Peterson graphs of other girths is still needed.

In the introduction we looked at some of the previous work around cops and robbers on generalized Petersen graphs. Our paper attempts to fill the notable gap in research on the cop numbers of such graphs that have girth 8. More research, however, is indicated for girth 8 and for all of the other possible girths. 

\begin{prob}
For each of the relationships between $n$ and $k$ listed in Corollary~\ref{cor:final}, what is $c(GP(n,k))$?
\end{prob}

As previously stated, it has been shown that the specific relationships between $n$ and $k$ listed in Corollary~\ref{cor:final} do not guarantee a cop number of 4. However, this in itself does not guarantee a cop number of 3, so any specific values for their cop numbers have not yet been proven.

In this study, we have applied a classic version of cops and robber, whereby the players have perfect information. 
However, you could easily take away this perfect information and analyse how this changes the cop number. Different rules can significantly affect game play and outcomes.

\begin{prob}
How is the game of cops and robbers affected (in this context) if the robber doesn't have perfect information?
\end{prob}

In addition to varying rules of play, there are also varying types of graphs upon which the game can be played. We examined only generalized Petersen graphs in this paper; however, many different graph families could be analyzed for their cop numbers. While there has been a great deal of research along these lines, many families remain unexplored. For example, one could look at the cop numbers for a particular family of snark graphs.

\begin{defn}
A snark is a connected, bridgeless, simple, cubic graph whose edge-chromatic number is $4$.
\end{defn}

Since the Petersen graph is a snark, and our structural result applies to cubic graphs, such research  would be closely related to this paper.

\begin{defn}
The flower snarks are an infinite family of snarks introduced by Rufus Isaacs \cite{isaacs}.
\end{defn}

\begin{prob}
What is the cop number of the flower snark $J_n$?
\end{prob}

More research could also be done on cop numbers for I-graphs. The family of I-graphs is a generalization of the family of generalized Petersen graphs, in which the jumps on the ``outer" cycle are also based on a parameter, rather than joining consecutive vertices. They are also cubic graphs that have girth at most $8$ (although they are not necessarily connected). Our results certainly apply to some I-graphs, although we have not investigated this in any detail. However, the upper bound found in \cite{ball} for the cop number of these graphs was $5$ rather than $4$, and nothing in our arguments helps to distinguish whether a cop number is $4$ or $5$ in a cubic graph.
The following question seems quite interesting:

\begin{prob}
When is the cop number of an I graph equal to 5?
\end{prob}

Cops and robbers, in all its variations, is a fun and interesting game to play on different graphs. There is much more research to be done on this topic and many more facets to explore.

\begin{appendix}
\section{Cop numbers, girths, and relationships for generalized Petersen graphs with $n \le 40$}
In these tables, the column ``relation" indicates any relationship(s) between $n$ and $k$ that are relevant to either the isomorphism class, the girth, or the cop number. The column ``iso?" indicates any smaller value of $k$ for which the graph in that row is isomorphic to $GP(n,k)$; it is empty if no such value exists. We also include a column ``source?" to indicate whether the cop number follows from theoretical results (including those in this paper), or has been found by computer. Some small cases  have also been verified by hand, but we list these as ``computer" unless they are part of a broader theoretical result.
\begin{table}[h!]
	\begin{center}
		\caption{Cop numbers, girth, and parameters for $GP(n,k)$, when $n \le 14$.}
		\begin{tabular}{r|r|c|c|r|r|l}
			\mathversion{bold}$n$ & \mathversion{bold}$k$ & \textbf{relation} & \textbf{iso?} & \mathversion{bold}$g$ & \mathversion{bold}$c$ & \textbf{source?} \\
			
			\hline\hline
			5 & 1 & $k=1$&&4&2 & theory, \cite{ball} \\ 
			5 & 2 & $k=2$ && 5 & 3 & theory, Proposition~\ref{prop:k=2}\\ \cline{1-7}
			6 & 1 & $k=1$ && 4&2 & theory, \cite{ball}\\ 
			6 & 2 & $n=3k$ && 3 &2 & computer \\ \cline{1-7}
			7 & 1 & $k=1$ && 4 & 2 & theory, \cite{ball} \\
			7 & 2 & $k=2$ && 5 & 3& theory, Proposition~\ref{prop:k=2} \\
			7 & 3 & $n=2k+1$& $k=2$& 5 & 3& theory, \cite{ball} \\ \cline{1-7}
			8 & 1 & $k=1$ && 4 & 2 & theory, \cite{ball} \\
			8 & 2 & $n=4k$ & & 4 & 2 & computer \\
			8 & 3 & $k=3$ && 6 & 3 & theory, \cite{ball} \\ \cline{1-7}
			9 & 1 & $k=1$ && 4 & 2 & theory, \cite{ball}\\
			9 & 2 & $k=2$ && 5 & 3& theory, Proposition~\ref{prop:k=2}\\
			9 & 3 & $n=3k$ && 3 & 2 & computer\\
			9 & 4 & $n=2k+1$ &$k=2$& 5 & 3 & theory, Proposition~\ref{prop:k=2}\\ \cline{1-7}
			10 & 1 & $k=1$ && 4 & 2 & theory, \cite{ball}\\
			10 & 2 & $k=2$, $n=5k$ && 5 & 3& theory, Proposition~\ref{prop:k=2}\\
			10 & 3 & $k=3$ && 6 & 3 & theory, \cite{ball}\\
			10 & 4 & $n=2k+2$ && 6 & 3 & computer\\ \cline{1-7}
			11 & 1 & $k=1$ && 4 & 2 & theory, \cite{ball}\\
			11 & 2 & $k=2$ && 5 & 3& theory, Proposition~\ref{prop:k=2}\\
			11 & 3 & $k=3$ && 6 & 3 & theory, \cite{ball}\\
			11 & 4 & $k=4$, $n=2k+3$ && 7 & 3 & computer\\ 			
			11 & 5 & $n=2k+1$ &$k=2$& 5 & 3 & theory, Proposition~\ref{prop:k=2}\\ \cline{1-7}			
			12 & 1 & $k=1$ && 4 & 2 & theory, \cite{ball}\\
			12 & 2 & $k=2$ && 5 & 3& theory, Proposition~\ref{prop:k=2}\\
			12 & 3 & $n=4k$ && 4 & 2 & computer\\
			12 & 4 & $n=3k$ && 3 & 3 & computer\\ 			
			12 & 5 & $n=2k+2$ && 6 & 3 & computer\\ \cline{1-7}			
			13 & 1 & $k=1$ && 4 & 2 & theory, \cite{ball}\\
			13 & 2 & $k=2$ && 5 & 3& theory, Proposition~\ref{prop:k=2}\\
			13 & 3 & $k=3$ && 6 & 3 & theory, \cite{ball}\\
			13 & 4 & $n=3k+1$ &$k=3$& 6 & 3 & theory, \cite{ball}\\ 			
			13 & 5 & $n=2k+3$&&7&3&computer\\ 
			13 & 6 & $n=2k+1$&$k=2$&5&3&theory, Proposition~\ref{prop:k=2}\\  \cline{1-7}
			14 & 1 & $k=1$ && 4 & 2 & theory, \cite{ball}\\
			14 & 2 & $k=2$ && 5 & 3& theory, Proposition~\ref{prop:k=2}\\
			14 & 3 & $k=3$ && 6 & 3 & theory, \cite{ball}\\
			14 & 4 & $n=3k+2$, $k=4$ && 7 & 3 & computer\\ 			
			14 & 5 & $n=3k-1$ &$k=3$ & 6 & 3 & theory, \cite{ball}\\ 
			14 & 6 & $n=2k+2$ & & 6 & 3 & computer\\
		\end{tabular}
	\end{center}
\end{table}

\begin{table}[h!]
	\begin{center}
		\caption{Cop numbers, girth, and parameters for $GP(n,k)$, when $15 \le n \le 19$.}
		\begin{tabular}{r|r|c|c|r|r|l}
			\mathversion{bold}$n$ & \mathversion{bold}$k$ & \textbf{relation} & \textbf{iso?} & \mathversion{bold}$g$ & \mathversion{bold}$c$ & \textbf{source?}  \\
			
			\hline \hline
			15 & 1 & $k=1$ && 4 & 2 & theory, \cite{ball}\\
			15 & 2 & $k=2$ && 5 & 3& theory, Proposition~\ref{prop:k=2}\\
			15 & 3 & $n=5k$, $k=3$ && 5 & 3 & theory, \cite{ball}\\
			15 & 4 & $k=4$ && 7 & 3 & computer\\ 			
			15 & 5 & $n=3k$ && 3 & 3 & computer\\ 
			15 & 6 & $n=2k+3$ && 7&3 & computer\\
			15 & 7 & $n=2k+1$ &$k=2$& 5&3 & theory, Proposition~\ref{prop:k=2}\\ \cline{1-7}
			16 & 1 & $k=1$ && 4 & 2 & theory, \cite{ball}\\
			16 & 2 & $k=2$ && 5 & 3& theory, Proposition~\ref{prop:k=2}\\
			16 & 3 & $k=3$ && 6 & 3 & theory, \cite{ball}\\
			16 & 4 & $n=4k$ && 4 & 3 & computer\\ 			
			16 & 5 & $n=3k+1$ &$k=3$& 6 & 3 & theory, \cite{ball}\\ 
			16 & 6 & $n=3k-2$ && 7 & 3 & computer\\ 
			16 & 7 & $n=2k+2$ && 6&3 & computer\\ \cline{1-7}
			17 & 1 & $k=1$ && 4 & 2 & theory, \cite{ball}\\
			17 & 2 & $k=2$ && 5 & 3& theory, Proposition~\ref{prop:k=2}\\
			17 & 3 & $k=3$ && 6 & 3 & theory, \cite{ball}\\
			17 & 4 & $k=4$ && 7 & 3 & computer\\ 			
			17 & 5 & $n=3k+2$ && 7 & 3 & computer \\
			17 & 6 & $n=3k-1$ &$k=3$& 6 & 3 & theory, \cite{ball}\\ 
			17 & 7 & $2n=5k-1$ &$k=5$& 8 & 3 & computer\\ 
			17 & 8 & $n=2k+1$ &$k=2$& 5&3 & theory, Proposition~\ref{prop:k=2}\\ \cline{1-7}
			18 & 1 & $k=1$ && 4 & 2 & theory, \cite{ball}\\
			18 & 2 & $k=2$ && 5 & 3& theory, Proposition~\ref{prop:k=2}\\
			18 & 3 & $k=3$, $n=6k$ && 6 & 3 & theory, \cite{ball}\\
			18 & 4 & $k=4$ && 7 & 3 & computer\\ 			
			18 & 5 & $n=3k+3$, $k=5$&& 8 & 3 & computer \\
			18 & 6 & $n=3k$ && 3 & 3 & computer\\ 
			18 & 7 & $2n=5k+1$ &$k=5$& 8 & 3 & computer \\ 
			18 & 8 & $n=2k+2$ && 6&3 & computer\\ \cline{1-7}
			19 & 1 & $k=1$ && 4 & 2 & theory, \cite{ball}\\
			19 & 2 & $k=2$ && 5 & 3& theory, Proposition~\ref{prop:k=2}\\
			19 & 3 & $k=3$ && 6 & 3 & theory, \cite{ball}\\
			19 & 4 & $k=4$ && 7 & 3 & computer\\ 			
			19 & 5 & $n=4k-1$ & $k=4$ & 7 & 3 & computer \\
			19 & 6 & $n=3k+1$ &$k=3$& 6 & 3 & theory, \cite{ball}\\ 
			19 & 7 & $n=3k-2$ && 7 & 3 & computer\\ 
			19 & 8 & $3n=7k+1$ &$k=7$& 7 & 3 & computer\\ 
			19 & 9 & $n=2k+1$ &$k=2$& 5&3 & theory, Proposition~\ref{prop:k=2}\\ 
		\end{tabular}
	\end{center}
\end{table}

\begin{table}[h!]
	\begin{center}
		\caption{Cop numbers, girth, and parameters for $GP(n,k)$, when $20 \le n \le 23$.}
		\begin{tabular}{r|r|c|c|r|r|l}
			\mathversion{bold}$n$ & \mathversion{bold}$k$ & \textbf{relation} & \textbf{iso?} & \mathversion{bold}$g$ & \mathversion{bold}$c$ & \textbf{source?}  \\
			
			\hline \hline
			20 & 1 & $k=1$ && 4 & 2 & theory, \cite{ball}\\
			20 & 2 & $k=2$ && 5 & 3& theory, Proposition~\ref{prop:k=2}\\
			20 & 3 & $k=3$ && 6 & 3 & theory, \cite{ball}\\
			20 & 4 & $n=5k$ && 5 & 3 & computer\\ 			
			20 & 5 & $n=4k$ & & 4 & 3 & computer \\
			20 & 6 & $n=3k+2$ && 7 & 3 & computer\\
			20 & 7 & $n=3k-1$ &$k=3$& 6 & 3 & theory, \cite{ball}\\ 
			20 & 8 & $n=2k+4$ && 8 & 3 & computer\\
			20 & 9 & $n=2k+2$ && 6&3 & computer\\ \cline{1-7}
			21 & 1 & $k=1$ && 4 & 2 & theory, \cite{ball}\\
			21 & 2 & $k=2$ && 5 & 3& theory, Proposition~\ref{prop:k=2}\\
			21 & 3 & $k=3$ && 6 & 3 & theory, \cite{ball}\\
			21 & 4 & $k=4$ && 7 & 3 & computer\\ 			
			21 & 5 & $n=4k+1$ &$k=4$& 7 & 3 & computer \\
			21 & 6 & $n=3k+3$ && 8 & 3 & computer\\
			21 & 7 & $n=3k$ && 3 & 3 & computer\\ 
			21 & 8 & $n=3k-3$ && 8 & 3 & computer\\
			21 & 9 & $n=2k+3$ && 7&3 & computer\\ 
			21 & 10 & $n=2k+1$ &$k=2$& 5&3 & theory, Proposition~\ref{prop:k=2}\\ \cline{1-7}
			22 & 1 & $k=1$ && 4 & 2 & theory, \cite{ball}\\
			22 & 2 & $k=2$ && 5 & 3& theory, Proposition~\ref{prop:k=2}\\
			22 & 3 & $k=3$ && 6 & 3 & theory, \cite{ball}\\
			22 & 4 & $k=4$ && 7 & 3 & computer\\ 			
			22 & 5 & $n=4k+2$, $k=5$ & & 8 & 3 & computer \\
			22 & 6 & $n=4k-2$ && 8 & 3 & computer\\
			22 & 7 & $n=3k+1$ &$k=3$& 6 & 3 & theory, \cite{ball}\\ 
			22 & 8 & $n=3k-2$ && 7 & 3 & computer\\
			22 & 9 & $2n=5k-1$ &$k=5$& 8&3 & computer\\ 
			22 & 10 & $n=2k+2$ && 6&3 & computer\\ \cline{1-7}
			23 & 1 & $k=1$ && 4 & 2 & theory, \cite{ball}\\
			23 & 2 & $k=2$ && 5 & 3& theory, Proposition~\ref{prop:k=2}\\
			23 & 3 & $k=3$ && 6 & 3 & theory, \cite{ball}\\
			23 & 4 & $k=4$ && 7 & 3 & computer\\ 			
			23 & 5 & $k=5$ & & 8 & 3 & computer \\
			23 & 6 & $n=4k-1$ &$k=4$& 7 & 3 & computer\\
			23 & 7 & $n=3k+2$ && 7 & 3 & computer\\ 
			23 & 8 & $n=3k-1$ &$k=3$& 6 & 3 & theory, \cite{ball}\\
			23 & 9 & $2n=5k+1$ &$k=5$& 8&3 & computer\\ 
			23 & 10 & $3n=7k-1$ &$k=7$& 7&3 & computer\\ 
			23 & 11 & $n=2k+1$ &$k=2$& 5&3 & theory, Proposition~\ref{prop:k=2}\\ 
		\end{tabular}
	\end{center}
\end{table}

\begin{table}[h!]
	\begin{center}
		\caption{Cop numbers, girth, and parameters for $GP(n,k)$, when $24 \le n \le 26$.}
		\begin{tabular}{r|r|c|c|r|r|l}
			\mathversion{bold}$n$ & \mathversion{bold}$k$ & \textbf{relation} & \textbf{iso?} & \mathversion{bold}$g$ & \mathversion{bold}$c$ & \textbf{source?}  \\
			
			\hline \hline
			24 & 1 & $k=1$ && 4 & 2 & theory, \cite{ball}\\
			24 & 2 & $k=2$ && 5 & 3& theory, Proposition~\ref{prop:k=2}\\
			24 & 3 & $k=3$ && 6 & 3 & theory, \cite{ball}\\
			24 & 4 & $n=6k$ && 6 & 3 & computer\\ 			
			24 & 5 & $k=5$ & & 8 & 3 & computer \\
			24 & 6 & $n=4k$ && 4 & 3 & computer\\
			24 & 7 & $n=3k+3$ && 8 & 3 & computer\\ 
			24 & 8 & $n=3k$ && 3 & 3 & computer\\
			24 & 9 & $n=3k-3$ && 8&3 & computer\\ 
			24 & 10 & $n=2k+4$ && 8&3 & computer\\ 
			24 & 11 & $n=2k+2$ && 6&3 & computer\\ \cline{1-7}
			25 & 1 & $k=1$ && 4 & 2 & theory, \cite{ball}\\
			25 & 2 & $k=2$ && 5 & 3& theory, Proposition~\ref{prop:k=2}\\
			25 & 3 & $k=3$ && 6 & 3 & theory, \cite{ball}\\
			25 & 4 & $k=4$ && 7 & 3 & computer\\ 			
			25 & 5 & $n=5k$ & & 5 & 3 & computer \\
			25 & 6 & $n=4k+1$ &$k=4$& 7 & 3 & computer\\
			25 & 7 &  && 8 & 4 & theory, \cite{ball} and Corollary~\ref{cor:final}, see also \cite{burgess}\\ 
			25 & 8 & $n=3k+1$ &$k=3$& 6 & 3 & theory, \cite{ball}\\
			25 & 9 & $n=3k-2$ && 7&3 & computer\\ 
			25 & 10 & $n=5k/2$ && 5&3 & computer\\ 
			25 & 11 & $4n=9k+1$ &$k=9$& 7&3 & computer\\ 
			25 & 12 & $n=2k+1$ &$k=2$& 5&3 & theory, Proposition~\ref{prop:k=2} \\ \cline{1-7} 
			26 & 1 & $k=1$ && 4 & 2 & theory, \cite{ball}\\
			26 & 2 & $k=2$ && 5 & 3& theory, Proposition~\ref{prop:k=2}\\
			26 & 3 & $k=3$ && 6 & 3 & theory, \cite{ball}\\
			26 & 4 & $k=4$ && 7 & 3 & computer\\ 			
			26 & 5 & $k=5$ & & 8 & 3 & computer \\
			26 & 6 & $n=4k+2$ && 8 & 3 & computer\\
			26 & 7 & $n=4k-2$ && 8 & 3 & computer\\ 
			26 & 8 & $n=3k+2$ && 7 & 3 & computer\\
			26 & 9 & $n=3k-1$ &$k=3$& 6&3 & theory, \cite{ball}\\ 
			26 & 10 &  && 8&4 &  theory, \cite{ball} and Corollary~\ref{cor:final}\\ 
			26 & 11 & $3n=7k+1$ &$k=7$& 8&3 & computer\\ 
			26 & 12 & $n=2k+2$ && 6&3 & computer\\
		\end{tabular}
	\end{center}
\end{table}

\begin{table}[h!]
	\begin{center}
		\caption{Cop numbers, girth, and parameters for $GP(n,k)$, when $27 \le n \le 29$.}
		\begin{tabular}{r|r|c|c|r|r|l}
			\mathversion{bold}$n$ & \mathversion{bold}$k$ & \textbf{relation} & \textbf{iso?} & \mathversion{bold}$g$ & \mathversion{bold}$c$ & \textbf{source?}  \\
			
			\hline \hline
			27 & 1 & $k=1$ && 4 & 2 & theory, \cite{ball}\\
			27 & 2 & $k=2$ && 5 & 3& theory, Proposition~\ref{prop:k=2}\\
			27 & 3 & $k=3$ && 6 & 3 & theory, \cite{ball}\\
			27 & 4 & $k=4$ && 7 & 3 & computer\\ 			
			27 & 5 & $k=5$ & & 8 & 3 & computer \\
			27 & 6 &  && 8 & 4 & theory, \cite{ball} and Corollary~\ref{cor:final}\\
			27 & 7 & $n=4k-1$ &$k=4$& 7 & 3 & computer \\ 
			27 & 8 & $n=3k+3$ && 8 & 3 & computer\\
			27 & 9 & $n=3k$ && 3&3 & computer\\ 
			27 & 10 & $3n=8k+1$ &$k=8$& 8&3 & computer\\ 
			27 & 11 & $2n=5k-1$ &$k=5$& 8&3 & computer\\ 
			27 & 12 & $n=2k+3$ && 7&3 & computer\\ 
			27 & 13 & $n=2k+1$ &$k=2$& 5&3 & theory, Proposition~\ref{prop:k=2} \\ \cline{1-7} 
			28 & 1 & $k=1$ && 4 & 2 & theory, \cite{ball}\\
			28 & 2 & $k=2$ && 5 & 3& theory, Proposition~\ref{prop:k=2}\\
			28 & 3 & $k=3$ && 6 & 3 & theory, \cite{ball}\\
			28 & 4 & $k=4$, $n=7k$ && 7 & 3 & computer\\ 			
			28 & 5 & $k=5$ & & 8 & 3 & computer \\
			28 & 6 &  && 8 & 4 &  theory, \cite{ball} and Corollary~\ref{cor:final}\\
			28 & 7 & $n=4k$ && 4 & 3 & computer\\ 
			28 & 8 & $n=7k/2$ && 7 & 4 & computer\\
			28 & 9 & $n=3k+1$ &$k=3$& 6&3 & theory, \cite{ball}\\ 
			28 & 10 & $n=3k-2$ && 7&3 &  computer\\ 
			28 & 11 & $2n=5k+1$ &$k=5$& 8&3 & computer\\ 
			28 & 12 & $n=2k+4$ && 8&3 & computer\\ 
			28 & 13 & $n=2k+2$ && 6&3 & computer\\ \cline{1-7}
			29 & 1 & $k=1$ && 4 & 2 & theory, \cite{ball}\\
			29 & 2 & $k=2$ && 5 & 3& theory, Proposition~\ref{prop:k=2}\\
			29 & 3 & $k=3$ && 6 & 3 & theory, \cite{ball}\\
			29 & 4 & $k=4$ && 7 & 3 & computer\\ 			
			29 & 5 & $k=5$ & & 8 & 3 & computer \\
			29 & 6 & $n=5k-1$ &$k=5$& 8 & 3 & computer\\
			29 & 7 & $n=4k+1$ &$k=4$& 7 & 3 & computer \\ 
			29 & 8 &  && 8 & 4 & theory, \cite{ball} and Corollary~\ref{cor:final}\\
			29 & 9 & $n=3k+2$ && 7&3 & computer\\ 
			29 & 10 & $n=3k-1$ &$k=3$& 6&3 & theory, \cite{ball}\\ 
			29 & 11 & $3n=8k-1$&$k=8$& 8&4 &  theory, \cite{ball} and Corollary~\ref{cor:final}\\ 
			29 & 12 & && 8&4 & theory, \cite{ball} and Corollary~\ref{cor:final}\\ 
			29 & 13 & $n=2k+3$ && 7&3 & computer\\ 
			29 & 14 & $n=2k+1$ &$k=2$& 5&3 & theory, Proposition~\ref{prop:k=2} \\ 
		\end{tabular}
	\end{center}
\end{table}

\begin{table}[h!]
	\begin{center}
		\caption{Cop numbers, girth, and parameters for $GP(n,k)$, when $30 \le n \le 31$.}
		\begin{tabular}{r|r|c|c|r|r|l}
			\mathversion{bold}$n$ & \mathversion{bold}$k$ & \textbf{relation} & \textbf{iso?} & \mathversion{bold}$g$ & \mathversion{bold}$c$ & \textbf{source?}  \\
			
			\hline \hline
			30 & 1 & $k=1$ && 4 & 2 & theory, \cite{ball}\\
			30 & 2 & $k=2$ && 5 & 3& theory, Proposition~\ref{prop:k=2}\\
			30 & 3 & $k=3$ && 6 & 3 & theory, \cite{ball}\\
			30 & 4 & $k=4$ && 7 & 3 & computer\\ 			
			30 & 5 & $n=6k$ & & 6 & 3 & computer \\
			30 & 6 & $n=5k$ && 5 & 3 &  computer\\
			30 & 7 & $n=4k+2$ && 8 & 3 & computer\\ 
			30 & 8 & $n=4k-2$ && 8 & 3 & computer\\
			30 & 9 & $n=3k+3$ && 8&3 & computer\\ 
			30 & 10 & $n=3k$ && 3&3 &  computer\\ 
			30 & 11 & $n=3k-3$ && 8&3 & computer\\ 
			30 & 12 & $n=5k/2$ && 5&3 & computer\\ 
			30 & 13 & $n=2k+4$ && 8&3 & computer\\ 
			30 & 14 & $n=2k+2$ && 6&3 & computer\\ \cline{1-7}
			31 & 1 & $k=1$ && 4 & 2 & theory, \cite{ball}\\
			31 & 2 & $k=2$ && 5 & 3& theory, Proposition~\ref{prop:k=2}\\
			31 & 3 & $k=3$ && 6 & 3 & theory, \cite{ball}\\
			31 & 4 & $k=4$ && 7 & 3 & computer\\ 			
			31 & 5 & $k=5$ & & 8 & 3 & computer \\
			31 & 6 & $n=5k+1$ &$k=5$& 8 & 3 & computer\\
			31 & 7 & && 8 & 4 & theory, \cite{ball} and Corollary~\ref{cor:final} \\ 
			31 & 8 &$n=4k-1$ &$k=4$& 7 & 3 & computer\\
			31 & 9 & $2n=7k-1$ &$k=7$& 8&4 &  theory, \cite{ball} and Corollary~\ref{cor:final}\\ 
			31 & 10 & $n=3k+1$ &$k=3$& 6&3 & theory, \cite{ball}\\ 
			31 & 11 & $n=3k-2$&& 7&3 &  computer\\ 
			31 & 12 & && 8&4 & theory, \cite{ball} and Corollary~\ref{cor:final}\\ 
			31 & 13 & $5n=12k-1$ &$k=12$& 8&4 &  theory, \cite{ball} and Corollary~\ref{cor:final}\\ 
			31 & 14 & $5n=11k+1$ &$k=11$& 7&3 & computer\\ 
			31 & 15 & $n=2k+1$ &$k=2$& 5&3 & theory, Proposition~\ref{prop:k=2} \\ 
		\end{tabular}
	\end{center}
\end{table}

\begin{table}[h!]
	\begin{center}
		\caption{Cop numbers, girth, and parameters for $GP(n,k)$, when $32 \le n \le 33$.}
		\begin{tabular}{r|r|c|c|r|r|l}
			\mathversion{bold}$n$ & \mathversion{bold}$k$ & \textbf{relation} & \textbf{iso?} & \mathversion{bold}$g$ & \mathversion{bold}$c$ & \textbf{source?}  \\
			
			\hline \hline
			32 & 1 & $k=1$ && 4 & 2 & theory, \cite{ball}\\
			32 & 2 & $k=2$ && 5 & 3& theory, Proposition~\ref{prop:k=2}\\
			32 & 3 & $k=3$ && 6 & 3 & theory, \cite{ball}\\
			32 & 4 & $k=4$ && 7 & 3 & computer\\ 			
			32 & 5 & $k=5$ & & 8 & 3 & computer \\
			32 & 6 &  && 8 & 4 &  theory, \cite{ball} and Corollary~\ref{cor:final}\\
			32 & 7 &  && 8 & 4 & theory, \cite{ball} and Corollary~\ref{cor:final}\\ 
			32 & 8 & $n=4k$ && 4 & 3 & computer\\
			32 & 9 & $2n=7k+1$ &$k=7$& 8&4 & theory, \cite{ball} and Corollary~\ref{cor:final}\\ 
			32 & 10 & $n=3k+2$ && 7&3 &  computer\\ 
			32 & 11 & $n=3k-1$ &$k=3$& 6&3 & theory, \cite{ball}\\ 
			32 & 12 &  && 8&4 & theory, \cite{ball} and Corollary~\ref{cor:final}\\ 
			32 & 13 & $2n=5k-1$ &$k=5$& 8&3 & computer\\ 
			32 & 14 & $n=2k+4$ && 8&3 & computer\\ 
			32 & 15 & $n=2k+2$ && 6&3 & computer\\ \cline{1-7}
			33 & 1 & $k=1$ && 4 & 2 & theory, \cite{ball}\\
			33 & 2 & $k=2$ && 5 & 3& theory, Proposition~\ref{prop:k=2}\\
			33 & 3 & $k=3$ && 6 & 3 & theory, \cite{ball}\\
			33 & 4 & $k=4$ && 7 & 3 & computer\\ 			
			33 & 5 & $k=5$ & & 8 & 3 & computer \\
			33 & 6 & && 8 & 4 &  theory, \cite{ball} and Corollary~\ref{cor:final}\\
			33 & 7 & && 8 & 4 & theory, \cite{ball} and Corollary~\ref{cor:final} \\ 
			33 & 8 &$n=4k+1$ &$k=4$& 7 & 3 & computer\\
			33 & 9 &  && 8&4 &  theory, \cite{ball} and Corollary~\ref{cor:final}\\ 
			33 & 10 & $n=3k+3$ && 8&3 & computer\\ 
			33 & 11 & $n=3k$&& 3&3 &  computer\\ 
			33 & 12 & $n=3k-3$ && 8&3 & computer\\ 
			33 & 13 & $2n=5k+1$ &$k=5$& 8&3 &  computer\\ 
			33 & 14 & $3n=7k+1$ &$k=7$& 8&4 &  theory, \cite{ball} and Corollary~\ref{cor:final}\\ 
			33 & 15 & $n=2k+3$ && 7&3 & computer\\ 
			33 & 16 & $n=2k+1$ &$k=2$& 5&3 & theory, Proposition~\ref{prop:k=2} \\ 
		\end{tabular}
	\end{center}
\end{table}

\begin{table}[h!]
	\begin{center}
		\caption{Cop numbers, girth, and parameters for $GP(n,k)$, when $34 \le n \le 35$.}
		\begin{tabular}{r|r|c|c|r|r|l}
			\mathversion{bold}$n$ & \mathversion{bold}$k$ & \textbf{relation} & \textbf{iso?} & \mathversion{bold}$g$ & \mathversion{bold}$c$ & \textbf{source?}  \\
			
			\hline \hline
			34 & 1 & $k=1$ && 4 & 2 & theory, \cite{ball}\\
			34 & 2 & $k=2$ && 5 & 3& theory, Proposition~\ref{prop:k=2}\\
			34 & 3 & $k=3$ && 6 & 3 & theory, \cite{ball}\\
			34 & 4 & $k=4$ && 7 & 3 & computer\\ 			
			34 & 5 & $k=5$ & & 8 & 3 & computer \\
			34 & 6 &  && 8 & 4 &  theory, \cite{ball} and Corollary~\ref{cor:final}\\
			34 & 7 & $n=5k-1$ &$k=5$& 8 & 3 & computer\\ 
			34 & 8 & $n=4k+2$ && 8 & 3 & computer\\
			34 & 9 & $n=4k-2$ && 8&3 & computer\\ 
			34 & 10 &  && 8&4 &   theory, \cite{ball} and Corollary~\ref{cor:final}\\ 
			34 & 11 & $n=3k+1$ &$k=3$& 6&3 & theory, \cite{ball}\\ 
			34 & 12 & $n=3k-2$ && 7&3 & computer\\ 
			34 & 13 &  && 8&4 &  theory, \cite{ball} and Corollary~\ref{cor:final}\\ 
			34 & 14 &  && 8&4 &  theory, \cite{ball} and Corollary~\ref{cor:final}\\ 
			34 & 15 & $4n=9k+1$ &$k=9$& 8&3 & computer\\ 
			34 & 16 & $n=2k+2$ && 6&3 & computer\\ \cline{1-7}
			35 & 1 & $k=1$ && 4 & 2 & theory, \cite{ball}\\
			35 & 2 & $k=2$ && 5 & 3& theory, Proposition~\ref{prop:k=2}\\
			35 & 3 & $k=3$ && 6 & 3 & theory, \cite{ball}\\
			35 & 4 & $k=4$ && 7 & 3 & computer\\ 			
			35 & 5 & $n=7k$ & & 7 & 3 & computer \\
			35 & 6 & && 8 & 4 &  theory, \cite{ball} and Corollary~\ref{cor:final}\\
			35 & 7 &$n=5k$ && 5 & 3 & computer \\ 
			35 & 8 & && 8 & 4 &  theory, \cite{ball} and Corollary~\ref{cor:final}\\
			35 & 9 & $n=4k-1$ &$k=4$& 7&3 & computer\\ 
			35 & 10 & $n=7k/2$ && 7&4 & computer\\ 
			35 & 11 & $n=3k+2$&& 7&3 &  computer\\ 
			35 & 12 & $n=3k-1$ &$k=3$& 6&3 & theory, \cite{ball}\\ 
			35 & 13 &$3n=8k+1$ &$k=8$& 8&4 &   theory, \cite{ball} and Corollary~\ref{cor:final}\\ 
			35 & 14 & $n=5k/2$ && 5&3 &  computer\\ 
			35 & 15 & $n=7k/3$ && 7&4 & computer\\ 
			35 & 16 & $5n=11k-1$ &$k=11$& 7&3 & computer\\ 
			35 & 17 & $n=2k+1$ &$k=2$& 5&3 & theory, Proposition~\ref{prop:k=2} \\ 
		\end{tabular}
	\end{center}
\end{table}

\begin{table}[h!]
	\begin{center}
		\caption{Cop numbers, girth, and parameters for $GP(n,k)$, when $36 \le n \le 37$.}
		\begin{tabular}{r|r|c|c|r|r|l}
			\mathversion{bold}$n$ & \mathversion{bold}$k$ & \textbf{relation} & \textbf{iso?} & \mathversion{bold}$g$ & \mathversion{bold}$c$ & \textbf{source?}  \\
			
			\hline \hline
			36 & 1 & $k=1$ && 4 & 2 & theory, \cite{ball}\\
			36 & 2 & $k=2$ && 5 & 3& theory, Proposition~\ref{prop:k=2}\\
			36 & 3 & $k=3$ && 6 & 3 & theory, \cite{ball}\\
			36 & 4 & $k=4$ && 7 & 3 & computer\\ 			
			36 & 5 & $k=5$ & & 8 & 3 & computer \\
			36 & 6 & $n=6k$ && 6 & 3 &  computer\\
			36 & 7 & $n=5k+1$ &$k=5$& 8 & 3 & computer\\ 
			36 & 8 &  && 8 & 4 &  theory, \cite{ball} and Corollary~\ref{cor:final}\\
			36 & 9 & $n=4k$ && 4&3 & computer\\ 
			36 & 10 &  && 8&4 &   theory, \cite{ball} and Corollary~\ref{cor:final}\\ 
			36 & 11 & $n=3k+3$ && 8&3 & computer\\ 
			36 & 12 & $n=3k$ && 3&3 & computer\\ 
			36 & 13 & $4n=11k+1$ &$k=11$& 8&3 & computer\\ 
			36 & 14 &  && 8&4 &  theory, \cite{ball} and Corollary~\ref{cor:final}\\ 
			36 & 15 &  && 8&4 &  theory, \cite{ball} and Corollary~\ref{cor:final}\\ 
			36 & 16 & $n=2k+4$ && 8&3 & computer\\ 
			36 & 17 & $n=2k+2$ && 6&3 & computer\\ \cline{1-7}
			37 & 1 & $k=1$ && 4 & 2 & theory, \cite{ball}\\
			37 & 2 & $k=2$ && 5 & 3& theory, Proposition~\ref{prop:k=2}\\
			37 & 3 & $k=3$ && 6 & 3 & theory, \cite{ball}\\
			37 & 4 & $k=4$ && 7 & 3 & computer\\ 			
			37 & 5 & $k=5$ & & 8 & 3 & computer \\
			37 & 6 & && 8 & 4 &  theory, \cite{ball} and Corollary~\ref{cor:final}\\
			37 & 7 & && 8 & 4 &  theory, \cite{ball} and Corollary~\ref{cor:final} \\ 
			37 & 8 & && 8 & 4 &  theory, \cite{ball} and Corollary~\ref{cor:final}\\
			37 & 9 & $n=4k+1$ &$k=4$& 7&3 & computer\\ 
			37 & 10 &  && 8&4 &  theory, \cite{ball} and Corollary~\ref{cor:final}\\ 
			37 & 11 &$3n=10k+1$ &$k=10$& 8&4 &   theory, \cite{ball} and Corollary~\ref{cor:final}\\ 
			37 & 12 & $n=3k+1$ &$k=3$& 6&3 & theory, \cite{ball}\\ 
			37 & 13 &$n=3k-2$ && 7&3 &   computer\\ 
			37 & 14 & $3n=8k-1$ &$k=8$& 8&4 &   theory, \cite{ball} and Corollary~\ref{cor:final}\\ 
			37 & 15 & $2n=5k-1$ &$k=5$& 8&3 & computer\\ 
			37 & 16 & $3n=7k-1$ &$k=7$& 8&4 &  theory, \cite{ball} and Corollary~\ref{cor:final}\\ 
			37 & 17 & $6n=13k+1$ &$k=13$& 7&3 & computer\\ 
			37 & 18 & $n=2k+1$ &$k=2$& 5&3 & theory, Proposition~\ref{prop:k=2} \\ 
		\end{tabular}
	\end{center}
\end{table}

\begin{table}[h!]
	\begin{center}
		\caption{Cop numbers, girth, and parameters for $GP(n,k)$, when $38 \le n \le 39$.}
		\begin{tabular}{r|r|c|c|r|r|l}
			\mathversion{bold}$n$ & \mathversion{bold}$k$ & \textbf{relation} & \textbf{iso?} & \mathversion{bold}$g$ & \mathversion{bold}$c$ & \textbf{source?}  \\
			
			\hline \hline
			38 & 1 & $k=1$ && 4 & 2 & theory, \cite{ball}\\
			38 & 2 & $k=2$ && 5 & 3& theory, Proposition~\ref{prop:k=2}\\
			38 & 3 & $k=3$ && 6 & 3 & theory, \cite{ball}\\
			38 & 4 & $k=4$ && 7 & 3 & computer\\ 			
			38 & 5 & $k=5$ & & 8 & 3 & computer \\
			38 & 6 &  && 8 & 4 &   theory, \cite{ball} and Corollary~\ref{cor:final}\\
			38 & 7 &  && 8 & 4 &  theory, \cite{ball} and Corollary~\ref{cor:final}\\ 
			38 & 8 &  && 8 & 4 &  theory, \cite{ball} and Corollary~\ref{cor:final}\\
			38 & 9 & $n=4k+2$ && 8&3 & computer\\ 
			38 & 10 &$n=4k-2$  && 8&3 & computer\\ 
			38 & 11 & $2n=7k-1$ &$k=7$& 8&4 &  theory, \cite{ball} and Corollary~\ref{cor:final}\\ 
			38 & 12 & $n=3k+2$ && 7&3 & computer\\ 
			38 & 13 & $n=3k-1$ &$k=3$& 6&3 & theory, \cite{ball}\\ 
			38 & 14 &  && 8&4 &  theory, \cite{ball} and Corollary~\ref{cor:final}\\ 
			38 & 15 & $2n=5k+1$ &$k=5$& 8&3 & computer\\ 
			38 & 16 &  && 8&4 &  theory, \cite{ball} and Corollary~\ref{cor:final}\\ 
			38 & 17 & $4n=9k-1$ &$k=9$& 8&3 & computer\\ 
			38 & 18 & $n=2k+2$ && 6&3 & computer\\ \cline{1-7}
			39 & 1 & $k=1$ && 4 & 2 & theory, \cite{ball}\\
			39 & 2 & $k=2$ && 5 & 3& theory, Proposition~\ref{prop:k=2}\\
			39 & 3 & $k=3$ && 6 & 3 & theory, \cite{ball}\\
			39 & 4 & $k=4$ && 7 & 3 & computer\\ 			
			39 & 5 & $k=5$ & & 8 & 3 & computer \\
			39 & 6 & && 8 & 4 &  theory, \cite{ball} and Corollary~\ref{cor:final}\\
			39 & 7 & && 8 & 4 &  theory, \cite{ball} and Corollary~\ref{cor:final} \\ 
			39 & 8 &$n=5k-1$ &$k=5$& 8 & 3 &  computer\\
			39 & 9 &  && 8&4 &  theory, \cite{ball} and Corollary~\ref{cor:final}\\ 
			39 & 10 & $n=4k-1$ &$k=4$& 7&3 & computer\\ 
			39 & 11 &$2n=7k+1$ &$k=7$& 8&4 &   theory, \cite{ball} and Corollary~\ref{cor:final}\\ 
			39 & 12 & $n=3k+3$ && 8&3 & computer\\ 
			39 & 13 &$n=3k$ && 3&3 &   computer\\ 
			39 & 14 & $n=3k-3$ && 8&3 &   computer\\ 
			39 & 15 &  && 8&4 &  theory, \cite{ball} and Corollary~\ref{cor:final}\\ 
			39 & 16 &  && 8&4 &  theory, \cite{ball} and Corollary~\ref{cor:final}\\ 
			39 & 17 & $7n=16k+1$ &$k=16$& 8&4 &  theory, \cite{ball} and Corollary~\ref{cor:final}\\ 
			39 & 18 & $n=2k+3$ && 7&3 & computer\\ 
			39 & 19 & $n=2k+1$ &$k=2$& 5&3 & theory, Proposition~\ref{prop:k=2} \\ 
		\end{tabular}
	\end{center}
\end{table}

\begin{table}[h!]
	\begin{center}
		\caption{Cop numbers, girth, and parameters for $GP(n,k)$, when $n=40$.}
		\begin{tabular}{r|r|c|c|r|r|l}
			\mathversion{bold}$n$ & \mathversion{bold}$k$ & \textbf{relation} & \textbf{iso?} & \mathversion{bold}$g$ & \mathversion{bold}$c$ & \textbf{source?}  \\
			
			\hline \hline
			40 & 1 & $k=1$ && 4 & 2 & theory, \cite{ball}\\
			40 & 2 & $k=2$ && 5 & 3& theory, Proposition~\ref{prop:k=2}\\
			40 & 3 & $k=3$ && 6 & 3 & theory, \cite{ball}\\
			40 & 4 & $k=4$ && 7 & 3 & computer\\ 			
			40 & 5 & $k=5$ & & 8 & 3 & computer \\
			40 & 6 &  && 8 & 4 &   theory, \cite{ball} and Corollary~\ref{cor:final}\\
			40 & 7 &  && 8 & 4 &  theory, \cite{ball} and Corollary~\ref{cor:final}\\ 
			40 & 8 & $n=5k$ && 5 & 3 &  computer\\
			40 & 9 &  && 8&4 & theory, \cite{ball} and Corollary~\ref{cor:final}\\ 
			40 & 10 &$n=4k$  && 4&3 & computer\\ 
			40 & 11 &  && 8&4 &  theory, \cite{ball} and Corollary~\ref{cor:final}\\ 
			40 & 12 &  && 8&4 & theory, \cite{ball} and Corollary~\ref{cor:final}\\ 
			40 & 13 & $n=3k+1$ &$k=3$& 6&3 & theory, \cite{ball}\\ 
			40 & 14 & $n=3k-2$ && 7&3 &  computer\\ 
			40 & 15 &  && 8&4 & theory, \cite{ball} and Corollary~\ref{cor:final}\\ 
			40 & 16 &$n=5k/2$  && 5&3 &  computer\\ 
			40 & 17 & $3n=7k+1$ &$k=7$& 8&4 & theory, \cite{ball} and Corollary~\ref{cor:final}\\ 
			40 & 18 & $n=2k+4$ && 8&3 & computer\\ 
			40 & 19 & $n=2k+2$ && 6&3 & computer\\ 
		\end{tabular}
	\end{center}
\end{table}

\end{appendix}

\end{document}